\documentclass[10pt,a4paper,reqno]{amsart}
\usepackage[english]{babel}
\usepackage{ngerman}		
\usepackage[latin1]{inputenc}
\usepackage{times}
\usepackage{amsmath, amssymb, amstext, amsthm, graphicx, enumerate, caption, lineno}
\usepackage{subcaption}
\makeindex
\usepackage{color} 
\definecolor{MyBlue}{rgb}{0,0.2,0.6}
\usepackage[linktocpage=true,colorlinks=true,linkcolor=blue,citecolor=blue, pdfstartview={FitH}]{hyperref}
\addto\extrasenglish{}
\addto\extrasenglish{}
\addto\extrasenglish{}
\newcommand{\R}{\mathbb{R}}

\DeclareMathOperator{\Nil}{Nil}
\DeclareMathOperator{\PSL}{PSL}
\DeclareMathOperator{\Iso}{Iso}
\DeclareMathOperator{\Sol}{Sol_3}

\DeclareMathOperator{\ext}{ext}
\DeclareMathOperator{\turn}{turn}
\DeclareMathOperator{\Stab}{Stab}

\theoremstyle{plain}
\newtheorem{Theorem}{Theorem}

\newtheorem{Lem}{Lemma}
\newtheorem{Conj}{Conjecture}

\newtheorem{Prop}{Proposition}
\theoremstyle{definition}
\newtheorem*{Def}{Definition}

\newtheorem{Rem}{Remark}
\newtheoremstyle{exercise}
  {8pt}                   
  {2pt}                   
  {\normalfont}           
  {}                      
  {\normalfont\bfseries}  
  {}                      
  {\newline}              
  {}
\theoremstyle{exercise}

\makeatletter
\renewcommand*{\c@Lem}{\c@Theorem}
\renewcommand*{\p@Lem}{\p@Theorem}

\renewcommand*{\c@Cor}{\c@Theorem}
\renewcommand*{\p@Cor}{\p@Theorem}

\renewcommand*{\c@Prop}{\c@Theorem}
\renewcommand*{\p@Prop}{\p@Theorem}

\renewcommand*{\c@Ex}{\c@Theorem}
\renewcommand*{\p@Ex}{\p@Theorem}

\renewcommand*{\c@Exer}{\c@Theorem}
\renewcommand*{\p@Exer}{\p@Theorem}

\renewcommand*{\c@Rem}{\c@Theorem}
\renewcommand*{\p@Rem}{\p@Theorem}

\renewcommand*{\c@Conj}{\c@Theorem}
\renewcommand*{\p@Conj}{\p@Theorem}

\newcommand{\mcH}{{\Small{MC}}$H$}

\newcommand{\mcone}{{\Small{MC}}$1$}
\newcommand{\Lcal}{\mathcal{L}}
\DeclareMathOperator{\sech}{sech}

\DeclareMathOperator{\atanh}{arctanh}
\newcommand{\N}{\mathbb{N}}
\newcommand{\Ekt}{\mathbb{E}(\kappa,\tau)}

\def\widebar{\accentset{{\cc@style\underline{\mskip10mu}}}}
\def\Widebar{\accentset{{\cc@style\underline{\mskip8mu}}}}
\makeatother

\begin{document}
\title[Invariant CMC cylinders in homogeneous three-manifolds]{Cylinders as Left Invariant CMC Surfaces in $\Sol$ and $E(\kappa,\tau)$-Spaces Diffeomorphic to $\R^3$}
\date{\today}
\author{Miroslav Vr\v{z}ina}
\address{Technische Universit\"at Darmstadt, Fachbereich Mathematik (AG 3 ``Geometrie und Approximation''),
	Schlossgartenstr.~7, 64289 Darmstadt, Germany}
\email{vrzina@mathematik.tu-darmstadt.de}

\begin{abstract}
	In the present paper we give a geometric proof for the existence of cylinders with constant mean curvature $H>H(X)$ in certain simply connected homogeneous three-manifolds $X$ diffeomorphic to $\R^3$, which always admit a Lie group structure. Here, $H(X)$ denotes the critical value for which constant mean curvature spheres in $X$ exist. Our cylinders are generated by a simple closed curve under a one-parameter group of isometries, induced by left translations along certain geodesics.  In the spaces $\Sol$ and $\widetilde{\PSL}_2(\R)$ we establish existence of new properly embedded constant mean curvature annuli. We include computed examples of cylinders in $\Sol$ generated by non-embedded simple closed curves.
\end{abstract}
\maketitle
\section*{Introduction}
Surfaces with constant mean curvature (for short \mcH-surfaces) are a classical topic in differential geometry. In the present paper we study invariant \mcH-surfaces whose mean curvature equation reduces to an ordinary differential equation. When this ODE has a simple closed curve as solution, we call the invariant surface generated by this curve an \mcH-\emph{cylinder}. We focus on various simply connected three-dimensional homogeneous manifolds, in which constant mean curvature surfaces have been studied recently, for example in \cite{Abresch}, \cite{Torralbo}, \cite{MP}, \cite{DanielMira} and \cite{MeeksSphere}. Our approach to this problem is as follows: We consider left invariant \mcH-surfaces generated by graphical curves. The invariance is with respect to left translations along a geodesic with certain symmetries. A comparison with \mcH-spheres yields properties of the graph which let us extend it to a simple closed solution curve by applying symmetries of the geodesic. 

In the first part of this paper we work in $\Sol$, a homogeneous space which has drawn particular attention in \cite{DanielMira}, since they settled existence and uniqueness of embedded \mcH-spheres in $\Sol$ (in combination with \cite{MeeksSphere}). It has only a three-dimensional isometry group and is thus the least symmetric homogeneous manifold among the Thurston geometries. A generic homogeneous three-manifold can have less symmetries, see \cite[Prop. 2.21. (4)]{MP}. For this reason, we find it instructive to start with $\Sol$ in order to understand which geometric properties and structures are needed to prove existence of \mcH-cylinders. Let us consider $\Sol$ a Riemannian fibration $\Sol\to\R$ with $\R^2$-fibres and base $\R$. At each point of $\Sol$ there are three distinguished geodesics which admit rotations of angle $\pi$: The base and two orthogonal lines in a $\R^2$-fibre, that is, one horizontal geodesic and two vertical geodesics. Since $\Sol$ is also a metric Lie group, left translations along any of these three geodesics define a one-parameter family of isometries. For each of these geodesics we construct surfaces which are invariant under the corresponding family of left translations. Without the need to state the ODE explicitly we can prove, in \autoref{thmcyl} and \autoref{thmcyl2}, that embedded horizontal and vertical \mcH-cylinders with $H>0$ exist. The \mcH-cylinders invariant by left translations along the base have been conjectured to exist by Lopez in \cite{Lopez14}, on the grounds of computed examples. We note that surfaces invariant under translation along orthogonal lines in a $\R^2$-fibre have not yet been considered,  and so we consider our families to be the first examples of properly embedded annuli with constant mean curvature $H>0$ in $\Sol$.

We also include images of computed examples of \mcH-cylinders in $\Sol$ which have the same invariance, but are only immersed (\autoref{immersed} and \autoref{bifurcation}). This class seems rich: We conjecture that there are infinitely many simple closed solution curves with self-intersections which generate such examples. This surprising phenomenon cannot occur in ambient spaces with higher dimensional isometry groups, for instance $\R^3$, where translations along and rotations about the same axis commute and thus imply rotational invariance of translationally invariant surfaces.

In a second part we consider the more standard Riemannian fibrations $E\to B$ over surfaces with geodesic fibres. They are parametrised as $E(\kappa,\tau)$-spaces with base curvature $\kappa$ and bundle curvature $\tau$. The $E(\kappa,\tau)$-spaces have $4$- or $6$-dimensional isometry groups. We restrict ourselves to $\kappa\leq0$ for two reasons: 1. The Riemannian product space $E(\kappa,0)=\mathbb{S}^2(\kappa)\times\R$ for $\kappa>0$ does not admit a Lie group structure, so that translations are not well-defined. 2. The Berger spheres $E(\kappa,\tau)$, where $\kappa>0$ and $\tau\neq0$, admit \mcH-spheres which are possibly self-intersecting; comparison spheres for our geometric approach are not available.  

In case of a $4$-dimensio\-nal isometry group, rotations about non-vertical geodesics need not be isometries, and their respective geodesic tubes need not have constant mean curvature. However, for $\kappa\leq0$ the $E(\kappa,\tau)$-spaces are also metric Lie groups, so that left translations along geodesics are isometries. A reasoning similar to the first part proves existence of \mcH-cylinders with ${H>H(E)}$, invariant under left translation along those geodesic axes which have a geodesic projection into the base space $B$, see \autoref{thmcylekt}. For $\kappa=-1$ and $\tau=0$ this includes tilted \mcH-cylinders in $\mathbb{H}^2\times\R$; we also get horizontal \mcH-cylinders in $\widetilde{\PSL}_2(\R)$. Again, we do not need to refer to the explicit form of the ODE. In \autoref{propweight} we calculate the horizontal diameter of these surfaces. The argument is based on a weight formula for \mcH-surfaces. We discuss possible generalisations and related open problems to conclude the paper.

\part{\texorpdfstring{\mcH}{mcH}-cylinders in \texorpdfstring{$\Sol$}{Sol}}\label{chapter1}
%
In Section 1 of this part we describe the metric Lie group $\Sol$ as a semi-direct product $\R^2\ltimes_A\R$, which fibres over $\R$ with $\R^2$-fibres. Section 2 is devoted to constant mean curvature surfaces invariant under left translations along the base: One problem concerns the ODE satisfied by a graph generating such a surface. The other problem is the geometric discussion of the ODE and the extension of the graphical solution to a simple closed embedded curve. For this class of surfaces we also include images of computed examples. In Section 3 we proceed analogously and construct \mcH-cylinders invariant under left translations along two special geodesics in a $\R^2$-fibre of $\Sol$.

\section{Preliminaries on \texorpdfstring{$\Sol$}{Sol}}\label{chapter11}

The space $\Sol$\index{manifold $\Sol$} is a simply connected homogeneous three-manifold diffeomorphic to $\R^3$ and as such a metric Lie group. We describe a model for this space and some properties.

\subsection{Model of \texorpdfstring{$\Sol$}{Sol3} and left translations}
We endow $\R^3$ with the Riemannian metric
\begin{equation}\langle\cdot,\cdot\rangle_{(x,y,z)}=e^{2z}dx^2+e^{-2z}dy^2+dz^2,\label{metric}\end{equation}
and set $\Sol:=\left(\R^3,\langle\cdot,\cdot\rangle\right)$. The multiplication
\begin{equation}(x_1,y_1,z_1)\ast(x_2,y_2,z_2):=\left(x_1+e^{-z_1}x_2,y_1+e^{z_1}y_2,z_1+z_2\right)\label{groupstructure}\end{equation}
turns $\Sol$ into a metric Lie group, i.e., for $a\in\Sol$ the left multiplication\index{left multiplication}\index{left-translation}
\[\Lcal_p\colon\Sol\to\Sol,\qquad \Lcal_p(g):=p\ast g\]
is an isometry of $\Sol$. This is the same model used in \cite{DanielMira}.

Via $q\colon \Sol\to\R,\,(x,y,z)\mapsto z$, we consider $\Sol$ also as a Riemannian fibration  with $\R^2$-fibres over the $z$-axis. This fibration decomposes $T_p\Sol$ orthogonally into the two-dimensional \emph{vertical space} $V_p=\operatorname{ker}(dq_p)$ and into the one-dimensional \emph{horizontal space} $H_p=\left(\operatorname{ker}(dq_p)\right)^\perp$. Note that this terminology is contrary to the standard notion in the model space $\R^3$ underlying $\Sol$.



Since $\Sol$ is a metric Lie group, \emph{left trans\-lation along the base} yields a one-parameter family of isometries $(\Phi_s)_{s\in\R}$,
\begin{equation}\Phi_s\colon\Sol\to\Sol\mbox{, }\Phi_s(x,y,z):=\Lcal_{(0,0,s)}(x,y,z)=(e^{-s}x,e^{s}y,z+s)\mbox{, } s\in\R.\label{htranslation}
\end{equation}
For $(x,y,z)\in\Sol$ with $(x,y)\neq(0,0)$ the orbit $\{(e^{-s}x,e^sy,0)\colon s\in\R\}$ traces out a hyperbola with $x$-axis and $y$-axis as asymptotes, and so $s\mapsto (e^{-s}x,e^sy,0)$ is called a hyperbolic rotation.

A two-parameter family of isometries $(\Psi_{a,b})_{a,b\in\R}$, \emph{left translations in the $\R^2$-fibres of $\Sol$}\index{translation in $\R^2$-fibre of $\Sol$}, is defined as $\Psi_{a,b}:=\Lcal_{(a,b,0)}\colon\Sol\to\Sol$, where
\begin{equation}
\Psi_{a,b}(x,y,z)=\Lcal_{(a,b,0)}(x,y,z)=\left(x+a,y+b,z\right),\qquad a,b\in\R.\label{vtranslation}
\end{equation}
In view of the fibration $\Sol\to\R$, we call $\Phi_s$ \emph{horizontal translation} and $\Psi_{a,b}$ \emph{vertical translation}. We note that horizontal and vertical translations do not commute. 

%

\subsection{Canonical frame and Riemannian connection}\label{canframe}At the origin let $(\partial_x,\partial_y,\partial_z)$ be the standard Euclidean frame. A left translation from the origin to $p=(x,y,z)$ gives the orthonormal frame
\begin{equation}E_1=e^{-z}\partial_x,\qquad E_2=e^{z}\partial_y,\qquad E_3=\partial_z.\label{onf}\end{equation}
The Riemannian connection with respect to this frame has the following representation:
\begin{equation}\begin{array}{lll}\nabla_{E_1}E_1=-E_3,&\nabla_{E_1}E_2=0,&\nabla_{E_1}E_3=E_1,\\
\nabla_{E_2}E_1=0,& \nabla_{E_2}E_2=E_3,&\nabla_{E_2}E_3=-E_2,\\
\nabla_{E_3}E_1=0,&\nabla_{E_3}E_2=0,& \nabla_{E_3}E_3=0.
\end{array} \label{connection}\end{equation}

\subsection{Special geodesics and induced isometries}We will make use of specific isometries, which we define first in general:
\begin{Def}
	Let $M$ be a Riemannian manifold with $\dim(M)\geq3$ and let $M_0$ be a submanifold of $M$ with  $\dim(M_0)=1$ or $\dim(M_0)=\dim(M)-1$. An isometry $\varphi\colon M\to M$ is said to be an \emph{involution} in $M_0$ if for every geodesic $\gamma$ intersecting $M_0$ perpendicularly at $\gamma(0)$ we have $\varphi(\gamma(t))=\gamma(-t)$.
\end{Def}
A fixed point set of an isometry is totally geodesic, so that in case $\dim(M_0)=1$ we can assume that $M_0$ is a geodesic $\Gamma$, and we refer to the inversion also as \emph{half-turn rotation about $\Gamma$} or \emph{rotation of angle $\pi$ about $\Gamma$}. For $\dim(M_0)=\dim(M)-1$ we call the inversion \emph{reflection through $M_0$}.

We consider the curves $c,c_\pm\colon\R\to\Sol$,\index{geodesics $\Sol$}
\begin{equation}
c(s)=(0,0,s)\qquad\text{and}\qquad
c_\pm(s)=(s,\pm s,0).\label{solgeodesics}
\end{equation}
In view of \eqref{onf} and \eqref{connection} they are constant speed geodesics. Moreover, $c$ is a horizontal geodesic and $c_\pm$ are orthogonal vertical geodesics. Using \eqref{htranslation} and \eqref{vtranslation}, they are orbits of horizontal and vertical translations due to
\[c(s)=\Phi_s(0,0,0)\qquad\mbox{and}\qquad c_\pm(s)=\Psi_{s,\pm s}(0,0,0),\qquad s\in\R.
\]
We introduce the notation
\begin{equation}\Psi_s^\pm:=\Psi_{s,\pm s}\overset{\eqref{vtranslation}}{=}\Lcal_{c_\pm(s)}\label{cpmtranslation},\end{equation}
and refer to it also as \emph{left translation along $c_\pm$}. Likewise, we call $\Phi_s=\Lcal_{(0,0,s)}=\Lcal_{c(s)}$ \emph{left translation along $c$}.



The reason why $c$ and $c_\pm$ are of particular interest is the following one. The stabiliser \begin{equation}\Stab_p:=\{\varphi\in\Iso(\Sol)\colon \varphi(p)=p\}\label{stab}\end{equation}
is known for all $p\in\Sol$. Due to homogeneity of $\Sol$, we concentrate on $p=(0,0,0)$. Then $\Stab_{(0,0,0)}$ is generated by the orientation-reversing isometries $\varphi_1,\varphi_2\in\Iso(\Sol)$,
\begin{equation}
\varphi_1(x,y,y)=(y,-x,-z)\qquad\mbox{and}\qquad \varphi_2(x,y,z)=(-x,y,z).\label{stabgen}
\end{equation}
From the form of the metric \eqref{metric} one can verify that $\varphi_1$ and $\varphi_2$ are isometries. As pointed out in \cite[Section 2.1]{DanielMira}, $\varphi_1$ has order $4$, $\varphi_2$ has order $2$ and \[\Stab_{(0,0,0)}=\{\varphi_2^m\circ\varphi_1^k\colon m=0,1\mbox{ and }k=0,1,2,3\}\] is isomorphic to the dihedral group $D_4$ with eight elements. The fixed point sets of these eight elements are either the point $\{(0,0,0)\}$, the geodesics $\{c(s)\colon s\in\R\}$ and $\{c_\pm(s)\colon s\in\R\}$, the $(y,z)$-plane $\{x=0\}$, or the $(x,z)$-plane $\{y=0\}$. Thus $c$ and $c_\pm$ are the only geodesics through $(0,0,0)$ admitting half-turn rotations, which, respectively, are given by $\rho,\rho_\pm\colon\Sol\to\Sol$,
\begin{align}
\rho(x,y,z)&=\varphi_1^2(x,y,z)=(-x,-y,z),\nonumber\\
\rho_+(x,y,z)&=(\varphi_2\circ\varphi_1^3)(x,y,z)=(y,x,-z),\label{solrotations}\\
\rho_-(x,y,z)&=(\varphi_2\circ\varphi_1)(x,y,z)=(-y,-x,-z).\nonumber
\end{align}
The reflections $\sigma_{1},\sigma_{2}\colon\Sol\to\Sol$ through $\{x=0\}$ and $\{y=0\}$ are
\begin{equation}
\sigma_1(x,y,z)=\varphi_2(x,y,z)=(-x,y,z)\quad\mbox{and}\quad\sigma_2(x,y,z)=(\varphi_2\circ\varphi_1^2)\overset{\eqref{solrotations}}{=}(x,-y,z).\label{solreflections}
\end{equation}
Left translation from $(0,0,0)$ to $p\in\Sol$ is an isometry, so that $\Stab_p$ is generated by $\Lcal_p\circ\varphi_1\circ\Lcal_p^{-1}$ and $\Lcal_p\circ\varphi_2\circ\Lcal_p^{-1}$. This implies that all $(x,z)$-planes and all $(y,z)$-planes are totally geodesic.



%

\subsection{Killing fields}Left translation along each of the coordinate axes defines a one-para\-meter group of isometries, generated by the following three Killing fields:
\begin{equation}
K_1=e^{z}E_1=\partial_x,\qquad K_2=e^{-z}E_2=\partial_y,\qquad
K_3=-xK_1+yK_2+E_3.\label{killing}
\end{equation}
\index{translation along $c$ in $\Sol$}
This is useful to describe Killing graphs defined on $\{y=0\}$. We will use this later.
\subsection{Constant mean curvature spheres}The existence, uniqueness and embeddedness of \mcH-spheres in $\Sol$ for values $H>1/\sqrt{3}$ has been settled by \cite{DanielMira}. Daniel and Mira introduced a novel way of dealing with \mcH-surfaces in $\Sol$. In combination with \cite{MeeksSphere}, their results also hold for $H>0$. We will need some specific properties of \mcH-spheres. For that purpose we recall the following definition:

\begin{Def}Let $X$ be a metric Lie group with identity element $e$ and left multiplication $\ell_x(y)=xy$. Let $f\colon\Sigma\to X$ be the immersion of a surface $\Sigma$ with unit normal vector field $N\colon\Sigma\to TX$ into the tangent bundle of $X$. The \emph{left invariant Gauss map of $f$} is the map $G\colon \Sigma\to\mathbb{S}^2=T_eX$, that assigns to each $p\in\Sigma$ the unit tangent vector $G(p)$ at the identity element $e$ given by $(d\ell_{f(p)})_e(G(p))=N_p$.
\end{Def}
We refer to \cite[Theorem 1.1]{MeeksSphere} for the following properties of \mcH-spheres $S_H$ in $\Sol$:
\begin{enumerate}
	\item For any $H>0$ there exists an embedded \mcH-sphere $S_H$ in $\Sol$. It is unique up to an ambient isometry.
	
	\item The left invariant Gauss map of $S_H$ is a diffeomorphism.
	
	\item There exists $p\in\Sol$ such that the ambient isometry group of $S_H$ is given by $\Stab_p$. The point $p$ is the centre of $S_H$ and we also write $S_H=S_H(p)$.
\end{enumerate}

Due to compactness, any \mcH-sphere $S_H(p)$ in $\Sol$ has a minimal and maximal value with respect to the coordinate axes. In order to simplify arguments used in Sections 2 and 3, we show that these values are attained on the respective axes through the centre of $S_H(p)$:

\begin{Prop}\label{solsphere}Let $H>0$ and $S_H$ be a sphere of constant mean curvature $H$
	in $\Sol$ with centre $p=(x_0,y_0,z_0)$. Then $\{x=x_0\}$ and $\{y=y_0\}$ are
	planes Alexandrov symmetry of $S_H$, and $S_H$ is a bi-graph with respect to each
	plane. The minimal and maximal values of the $x$, $y$ and
	$z$ coordinates arise on the curves $s\mapsto (x_0+s,y_0,z_0)$, $s\mapsto (x_0,y_0+s,z_0)$ and $s\mapsto (x_0,y_0,z_0+s)$ through $p=(x_0,y_0,z_0)$, respectively.
\end{Prop}

\begin{proof}The claim about the Alexandrov symmetry planes follows from \cite[Propostion 5.4.]{DanielMira} and item 3. above.
	
	In order to prove the last claim, we use the Transversality Lemma \cite[Lemma 3.1]{MeeksTransversality}. Since the left invariant Gauss map of $S_H(p)$ is a diffeomorphism, we only need to exhibit suitable two-dimensional subgroups $\Sigma$ of $\Sol$. In view of \eqref{groupstructure}, it is easy to check that
	\[\Sigma_1=\{x=0\},\qquad \Sigma_2=\{y=0\}\qquad\mbox{and}\qquad\Sigma_3=\{z=0\}\]
	are two-dimensional subgroups of $\Sol$. 
	
	Let us give the details for the $x$ coordinate. The quotient space \[\Sol/\Sigma_1=\{g\Sigma_1\colon g\in\Sol\}\] consisting of the left cosets of $\Sigma_1$ is the foliation of $\Sol$ by $(y,z)$-planes, which are orthogonal to $s\mapsto (x_0+s,y_0,z_0)$. By the Transversality Lemma, the set of left cosets intersecting $S_H(p)$ can be parametrised by $[0,1]$, that is, we have \[S_H(p)\cap \Sol/\Sigma_1=\{g(t)\Sigma_1\colon t\in[0,1]\},\] and each of the left cosets $g(0)\Sigma_1$ and $g(1)\Sigma_1$ intersects $S_H(p)$ in a single point $p(0)$, respectively $p(1)$. We claim that $p(0)$ and $p(1)$ lie on $s\mapsto(x_0+s,y_0,z_0)$. We note that reflection through $\{y=y_0\}$ leaves $S_H(p)$ and each left coset $g\Sigma_1$ invariant (since the latter are $(y,z)$-planes). If $p(0)$ and $p(1)$ were on $s\mapsto(x_0+s,\widetilde{y_0},\widetilde{z_0})$ with $\widetilde{y_0}\neq y_0$ or $\widetilde{z_0}\neq z_0$, then reflection through $\{y=y_0\}$ would result in a second point intersecting $g(0)\Sigma_1$ and $g(1)\Sigma_1$, which contradicts the Transversality Lemma. The minimal and maximal values of the $x$ coordinate are the corresponding coordinates of $p(0)$ and $p(1)$.
	
	For the $y$ coordinate, the proof is analogous to the one for the $x$ coordinate. For the $z$ coordinate, instead of a reflection through a plane, one argues using the half-turn rotation $\rho_{x_0,y_0}$ about $s\mapsto(x_0,y_0,z_0+s)$. Both $S_H(p)$ and the foliation by $(x,y)$-planes, which is the quotient space $\Sol/\Sigma_3$ in this case, are invariant under the half-turn rotation $\rho_{x_0,y_0}$.
\end{proof}


\section{Horizontal \texorpdfstring{\mcH}{mcH}-cylinders in \texorpdfstring{$\Sol$}{Sol}}\label{chapter12}
In this section we study constant mean curvature surfaces invariant under left translation along the base $c$ of $\Sol$. First we describe properties of the differential equation for constant mean curvature surfaces invariant by $(\Phi_s)_{s\in\R}$, as defined in \eqref{htranslation}. These are natural implications by the geometry of $\Sol$. Then we discuss the solution of this ODE geometrically. We use the maximum principle to derive properties which let us extend the respective graph by reflections to an embedded closed solution curve. Since $c$ is a horizontal geodesic, we call the surfaces \emph{horizontal \mcH-cylinders}. We also discuss further solutions which we obtained computationally.

\subsection{ODE for surfaces invariant under left translations along \texorpdfstring{$c$}{c}}The foliation by $(x,y)$-planes of $\Sol$ is invariant under left translation along $c$. Therefore it is sufficient to consider a curve in the fibre \[S_0:=\{(x,y,z)\in\Sol\colon z=0\}\]
as generating curve of a surface invariant by left translation along $c$.

Explicitly, for $\mathcal{C}^2$-functions $x\colon J\to\R$ and $y\colon J\to\R$, defined on an open interval $J\subset\R$, the curve
\[\gamma\colon J\to\Sol,\qquad \gamma(t):=(x(t),y(t),0)\]
is in the fibre $S_0$ and the invariant surface generated by left translation of $\gamma$ along $c$ is parametrised by
\begin{equation}f\colon \R\times J\to\Sol,\qquad f(s,t):=\Phi_s(\gamma(t))=\left(e^{-s}x(t),e^sy(t),s\right),\label{surface}\end{equation}\index{invariant surface in $\Sol$}
where $\Phi_s$ is as in \eqref{htranslation}. The mean curvature $H$ of $f$ is independent of the parameter $s$, that is, we have $H=H(t)$. Requiring $H$ to be constant leads to an ordinary differential equation for the curve $\gamma$. 
Such surfaces were studied in \cite{Lopez14}, too. For $H=0$ some initial value problems have explicit solutions or allow for qualitative discussions involving first integrals. For $H>0$, however, the mean curvature equation appears too complicated for these approaches.

We will consider curves $\gamma(t)=(t,h(t),0)$, which are Killing graphs with respect to the Killing field $K_2=\partial_y$ from \eqref{killing}. For such graphs the ODE can be described as follows:
\begin{Prop}\label{odesol2}Let $H\in\R$.
	\begin{enumerate}[(a)]
		\item There is a smooth function $F\colon\R^3\to\R$ such that the invariant surface
		\[
		f\colon\R\times J\to\Sol,\qquad f(s,t):=\Phi_s(t,h(t),0),\qquad\text{where }h\in\mathcal{C}^2(J,\R),
		\]
		has constant mean curvature $H$ with respect to the inner normal if and only if 
		\begin{equation}h^{\prime\prime}(t)=F(t,h(t),h^\prime(t))\qquad \text{for all t }\in J.\label{odeprop}\end{equation}
		
		\item The invariant surface $\tilde{f}\colon\R\times J\to\Sol\mbox{, }\tilde{f}(s,t):=\Phi_{-s}(h(t),t,0)$ has constant mean curvature $H$ if and only if $h\in\mathcal{C}^2(J,\R)$ satisfies \eqref{odeprop}, i.e., $x$-graphs and $y$-graphs as generating curves of invariant surfaces with constant mean curvature $H$ satisfy the same ODE.
	\end{enumerate}\end{Prop}
	
	\begin{proof}\begin{enumerate}[(a)]
			\item Let $v_1:=\partial_sf$ and $v_2:=\partial_tf$. We denote the inner normal to $f$ by $N$, so that $g_{ij}:=\langle v_i,v_j\rangle$ and $b_{ij}:=\langle \nabla_{v_i}v_j,N\rangle$ for $i,j\in\{1,2\}$ are the coefficients of the first and second fundamental form. Then the mean curvature of $f$ is given by
			\[
			H=\frac{b_{11}g^{11}+2b_{12}g^{12}+b_{22}g^{22}}{2}.
			\]
			We have
			\[
			v_2=E_1+h^\prime E_2\qquad\text{and}\qquad \nabla_{v_2}v_2=\underbrace{\nabla_{v_2}E_1+h^\prime\nabla_{v_2}E_2}_{=:w}+h^{\prime\prime}E_2
			\]
			Here we note that $H$ depends on $t$, $h(t)$, $h^\prime(t)$ and $h^{\prime\prime}(t)$.
			
			We assume $H$ to be constant and therefore get an implicit differential equation depending on $h^\prime$ and $h^{\prime\prime}$. Now we want to show that we can solve this implicit equation for $h^{\prime\prime}$.
			
			Obviously $w$ is independent of $h^{\prime\prime}$ and the only term containing $h^{\prime\prime}$ is
			\[
			\frac{b_{22}g^{22}}{2}=\frac{b_{22}g_{11}}{2\det(g)}= \frac{\langle \nabla_{v_2}v_2,N\rangle g_{11}}{2\det(g)}=\frac{\left(h^{\prime\prime}\langle N,E_2\rangle+\langle w,N\rangle\right)g_{11}}{2\det(g)}.
			\]
			The invariant surface $f$ is a Killing graph with respect to the Killing field $K_2=e^{-s}E_2$, see \eqref{killing}, so that $\langle N,E_2\rangle$ is positive, because $N$ is chosen as inner normal. We also have $g_{11}>0$ because the Killing field generated by left translation along $c$ is non-trivial. Therefore we can solve the implicit equation for $h^{\prime\prime}$ and get a function ${F\colon\R^3\to\R}$ with $h^{\prime\prime}(t)=F(t,h(t),h^\prime(t))$. The function $F$ is smooth because each $\Phi_s$ is smooth and so are $g$ and $b$. It is defined on all of $\R^3$ because we can prescribe any kind of function $h\colon J\to\R$.
			
			\item The equation $\Phi_{-s}\circ\rho_+=\rho_+\circ\Phi_s$, verified using \eqref{htranslation} and \eqref{solrotations}, implies $\tilde{f}=\rho_+\circ f$, i.e. $\tilde{f}$ and $f$ are isometric. Thus the claim about the ODE follows from (a).\qedhere
		\end{enumerate}
	\end{proof}
	
	\subsection{Half-cylinder solution and its extension to \texorpdfstring{\mcH}{mcH}-cylinders with axis \texorpdfstring{$c$}{c}}\label{chapter122}We consider the ODE for surfaces invariant by left translation along the base $c$ first. We can apply the Picard-Lindel\"of Theorem to \eqref{odeprop} because $F$ is smooth. We obtain a maximal solution $h$. For constant mean curvature $H>0$ the maximum principle yields some general properties by comparing the surface $f$ with \mcH-spheres:
	
	\begin{Prop}\label{lemode}Let $a\in\R$ and $H>0$. Then there exists a unique maximal solution $h\colon I_{\max}\to\R$ with $h(0)=a$ and $h^\prime(0)=0$ satisfying \eqref{odeprop}. It has the following properties:
		\begin{enumerate}[(a)]
			\item \emph{[Symmetry]:} There is $R=R(a)\in(0,\infty]$ such that we have $I_{\max}=(-R,R)$ and $h(t)=h(-t)$ for all $t\in(-R,R)$.
			
			
			\item \emph{[Bounded existence interval]:} We have $R(a)<\infty$.
			
			\item \emph{[Bounded graph]:} There is $K=K(a)>0$ such that we have $\lvert h(t)\rvert\leq K$ for all $t\in[-R,R]$ where $h(\pm R):=\lim_{t\to\pm R}h(t)$.\end{enumerate}
		

		
	\end{Prop}
	
	\begin{proof}Let $h\colon I_{\max}\to\R$ be the unique maximal solution of $h^{\prime\prime}(t)=F(t,h(t),h^\prime(t))$ with $h(0)=a$ and $h^\prime(0)=0$. We recall that
		\[f(s,t)=\Phi_s(t,h(t),0)=(e^{-s}t,e^sh(t),s)\]
		is the surface invariant under the left translations $(\Phi_s)_{s\in\R}$ and generated by $(t,h(t),0)$. We denote its image by $\Sigma$. We will use frequently that left translations along the $x$-axis or $y$-axis are isometries of $\Sol$; these translations are $\Psi_{u,0}$ and $\Psi_{0,v}$, as defined in \eqref{vtranslation}. From \autoref{solsphere} we only use that \mcH-spheres have a centre, are embedded and compact.
		
		\begin{figure}[ht]
			\begin{center}			
				\includegraphics{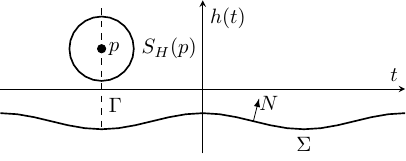}
			\end{center}
			\caption{\autoref{lemode} (b): Comparison argument indicating that the case $R=\infty$ is impossible.}
			\label{alemma}
		\end{figure}
		
		(a): Looking at \eqref{htranslation} and \eqref{solreflections}, it is easy to check $\sigma_{1}\circ\Phi_s=\Phi_s\circ\sigma_{1}$, so that a reflection of the solution through $\{x=0\}$ gives another solution of the same ODE. The initial values $h(0)=a$ and $h^\prime(0)=0$ are invariant under $\sigma_{1}$, i.e., we obtain the same solution. This proves $I_{\max}=(-R,R)$ for some $R\in(0,\infty]$ and $h(-t)=h(t)$ for all $t\in(-R,R)$.
		
		(b): Assume $R=\infty$ and let $\Pi_y\colon\Sol\to\R^2$ be defined by $\Pi_y(x,y,z):=(x,z)$.
		
		Due to our assumption on $R$ and the compactness of \mcH-spheres in $\Sol$, there exists $p\in\Sol$ such that for the \mcH-sphere $S_H(p)$ centred at $p$ we have 
		
		\begin{enumerate}[1.]\item $\Pi_y(S_H(p))\subset\Pi_y(\Sigma)$ and
			\item $S_H(p)$ is contained in the mean convex side of $\Sigma$.
		\end{enumerate}
		
		Since \mcH-spheres are unique up to an ambient isometry, moving spheres by left translation along some $y$-axis $\Gamma$ through $p$ inside the mean convex side of $\Sigma$ towards the surface leads to a first tangential contact point of in the interior of $\Sigma$. Because of $R=\infty$ the surface has no boundary. The maximum principle then shows $\Sigma=S_H$, which is a contradiction. See \autoref{alemma} above.
		
		(c): Suppose this were false. As a solution of an ODE, $h$ can only be unbounded at the boundary of $I_{\max}=(-R,R)$. The solution is symmetric, so that we only consider the cases $\lim_{t\to R} h(t)=\pm\infty$. Both cases are ruled out by moving \mcH-spheres along an $x$-axis towards $f(\R\times(0,R))$; compare with \autoref{blemma} above.	
		\begin{figure}
			\begin{center}			
				\includegraphics{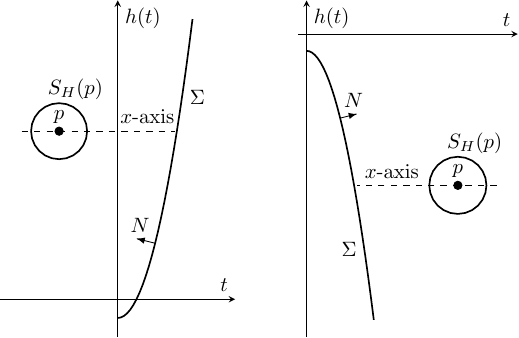}
			\end{center}
			\caption{Geometry for \autoref{lemode} (c): On the left side we see the case $\lim_{t\to R}h(t)=+\infty$ and on the right side $\lim_{t\to R}h(t)=-\infty$.}
			\label{blemma}
		\end{figure}
	\end{proof}

	We are interested in a particular solution of the ODE, see \autoref{zeroheightillustration}:
	\begin{Lem}[Height zero half-cylinder solution]\label{zeroheight}There is $a_0<0$ such that the maximal solution with $h(0)=a_0$ and $h^\prime(0)=0$ has the following properties:
		\begin{enumerate}[(a)]
			\item \emph{[Height zero and existence interval]:} We have $h(\pm R(a_0))=0$ and $R(a_0)=-a_0$.
			\item \emph{[Below height zero]:} For all $t\in(-R(a_0),R(a_0))$ we have $h(t)<0$.
			\item \emph{[Asymptotic behaviour]:} The solution satisfies $\lim_{t\to \pm R}h^{\prime}(t)=\pm\infty$.
		\end{enumerate}
		
	\end{Lem}
	\begin{figure}
		\begin{center}
			\includegraphics{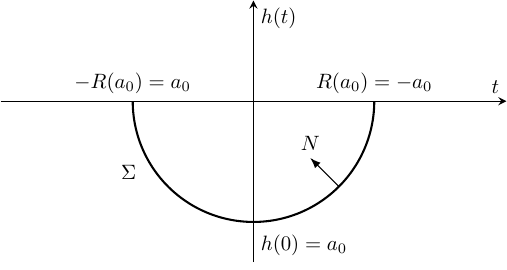}
		\end{center}
		\caption{Illustration of \autoref{zeroheight}}
		\label{zeroheightillustration}
	\end{figure}
	\begin{proof}(a): The function $\varphi\colon\R\to\R\mbox{, }\varphi(a):=h(R(a))$ is continuous. For $a=0$ we claim $\varphi(0)>0$. We argue by contradiction and suppose $\varphi(0)\leq0$. Then the boundary of $\Sigma$ is contained in $\{y\leq0\}$ because of
		\[f(s,R(0))=\Phi_s(R(0),h(R(0)),0)=\Phi_s(R(0),\varphi(0),0)=(e^{-s}R(0),e^s\varphi(0),s).\]
		Since $a=0$, this means we can move an \mcH-sphere $S_H$ centred at $(0,v,0)$ for some $v>0$ along $[0,\infty)\to\Sol\mbox{, }s\mapsto(0,v-s,0)$ towards $\Sigma$ without touching the boundary of it, i.e.,  we get a point of tangential contact. This is because the boundary of $\Sigma$ is contained in $\{y\leq 0\}$ and \mcH-spheres attain its minimal value on the $y$-axis through the centre of the sphere (see \autoref{solsphere}). The maximum principle shows $\Sigma=S_H$, a contradiction. This proves $\varphi(0)>0$. See \autoref{zeroheighti}.
		
		\begin{figure}[ht]
			\begin{center}
				\includegraphics{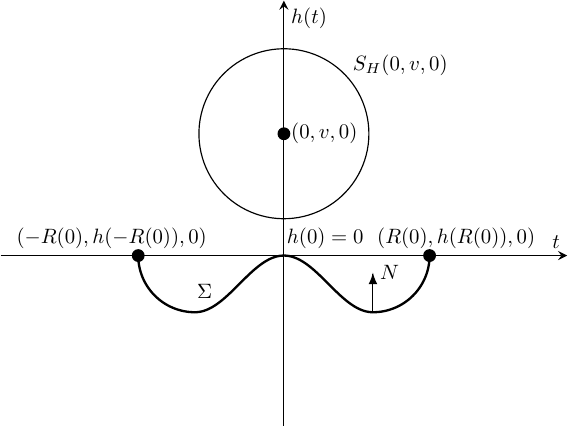}
			\end{center}
			\caption{\autoref{zeroheight} (a): Illustrated is the comparison argument to rule out $h(R(0))\leq0$.}
			\label{zeroheighti}
		\end{figure}

		Next we show there is $a<0$ with $\varphi(a)<0$. If we had $\varphi(a)\geq0$ for all $a\leq0$, then we could find $\tilde{a}<0$ such that it were possible to move an \mcH-sphere along an $x$-axis to the surface $f|_{\R\times(0,R(\tilde{a}))}$ without touching its boundary, a contradiction. So there is some $a<0$ with $\varphi(a)<0$, and by the intermediate value theorem there is $a_0<0$ with $\varphi(a_0)=0$.
		
		To show $R(a_0)=-a_0$, let us consider $\tilde{f}:=\rho_+\circ f$. Then \autoref{odesol2} (b) implies that $h$ is also a height zero solution to the initial values $h(0)=-R(a_0)$ and $h^\prime(0)=0$ with $I_{\max}=(a_0,-a_0)$. This can only hold for $R(a_0)=-a_0$.
		
		(b): Assume this were false. Then the solution $h$ would attain its maximum at some $t_0\in(-R(a_0),R(a_0))$ with $h(t_0)\geq0$. Therefore we can move an \mcH-sphere centred at $(t_0,v,0)$ for some $v>0$ along $[0,\infty)\to\Sol\mbox{, }s\mapsto(t_0,v-s,0)$ towards $\Sigma$ without touching its boundary. As above, this follows from $\partial\Sigma\subseteq\{y=0\}$ and \mcH-spheres attaining their minimal value on the $y$-axis through the centre of the sphere (see \autoref{solsphere}). We thus obtain a point of tangential contact with $\Sigma$, which implies $\Sigma=S_H$, a contradiction.
		
		(c): Let $\tilde{F}\colon\R^3\to\R^2\mbox{, }\tilde{F}(\tau,\xi,\eta):=(\eta,F(\tau,\xi,\eta))$, so that the maximal solution of $u^\prime(t)=\tilde{F}(t,u(t))$ with $u(0)=(a_0,0)$ is given by $v(t):=(h(t),h^\prime(t))$.
		
		We know the phase space of $\tilde{F}$ is $\R^3$. General ODE theory implies that
		\[I_{\max}\to\R^2,\qquad t\mapsto (t, v(t))= (t,h(t),h^\prime(t))\]
		leaves every compact subset in $\R^3$, in particular \[[-R(a_0),R(a_0)]\times[-K(a_0),K(a_0)]\times[-C,C]\] for every $C>0$, where $K(a_0)$ is as defined in \autoref{lemode} (b). These properties imply $\lim_{t\to \pm R}\left\lvert h^\prime(t)\right\rvert=\infty$. The sign $\lim_{t\to\pm R}h(t)=\pm\infty$ follows from (b).\end{proof}		
	\begin{Rem}
		The arguments in the proof where we use \autoref{solsphere} work also without knowing exactly where the minimal or maximal values are attained at. However, one then has to be more careful when moving \mcH-spheres towards the invariant surface $\Sigma$. It is then necessary to make use of left translations along $x$-axes and use the explicit form \eqref{vtranslation}.
	\end{Rem}

	We use one height zero solution to obtain a smooth embedded closed curve $\gamma$ generating a left invariant cylinder $f$ with constant mean curvature $H>0$ and get our first main result:
	
	\begin{Theorem}\label{thmcyl}Consider the fibration $\Sol\to\R\mbox{, }(x,y,z)\mapsto z$. Then for each $H>0$ there is a smooth embedded simple closed curve $\gamma$ in a $\R^2$-fibre which generates a $(\Phi_s)_{s\in\R}$-invariant embedded surface $f(s,t)=\Phi_s(\gamma(t))$ in $\Sol$ with constant mean curvature $H$. The surface is invariant by $\Stab_{(0,0,0)}$, a dihedral group of order $8$ generated by $\{\sigma_1,\sigma_2,\rho_\pm\}$.
	\end{Theorem}
	For the notation, recall that $(\Phi_s)_{s\in\R}$ denotes the group of left translations along the base $c$ in the model of $\Sol$ described in \autoref{chapter11}; see \eqref{stab}, \eqref{solrotations} and \eqref{solreflections} for the definitions of $\Stab_p$, $\rho_\pm$ and $\sigma_1$ or $\sigma_2$, respectively.
	In the following we refer to these surfaces as \emph{\mcH-cylinders with axis $c$ in $\Sol$} or, since $c$ is a horizontal geodesic, as \emph{horizontal \mcH-cylinders in $\Sol$}.\index{cylinder in $\Sol$} We have computed an example, see \autoref{imodegraph}, and details of the computation are explained in \autoref{cylindercomputed}.
	\begin{figure}
		\begin{center}
			\includegraphics[width=1\linewidth]{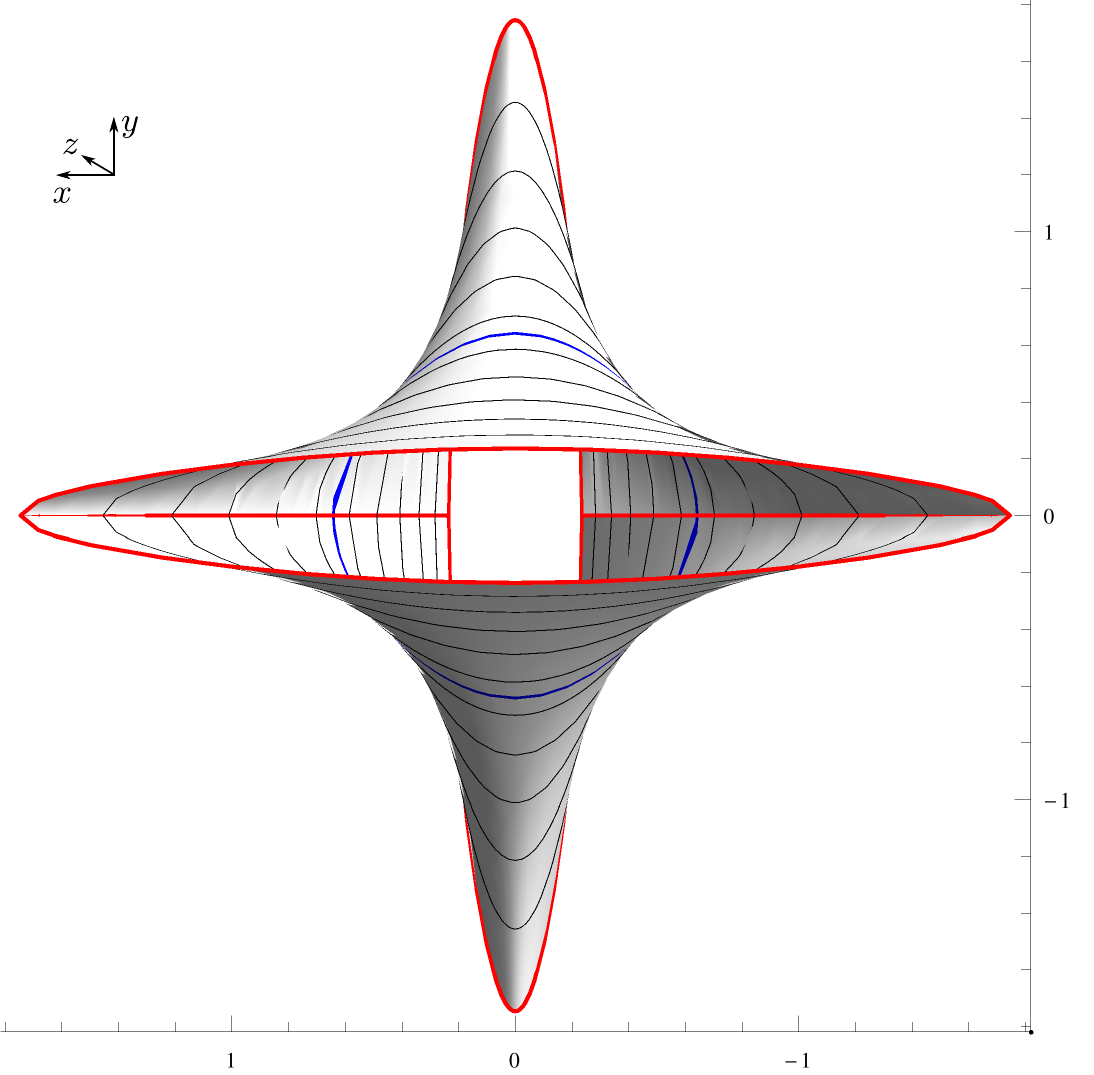}
		\end{center}
		%
		\caption{Computed example of an \mcone-cylinder in $\Sol$ established in \autoref{thmcyl}. All level lines shown are the intersection of the cylinder with a $\R^2$-fibre of the fibration $\Sol\to\R\mbox{, }(x,y,z)\mapsto z$. The level lines are isometric. The curve in blue is contained in the fibre $S_0=\{z=0\}$ and generates the \mcone-cylinder.}
		\label{imodegraph}
	\end{figure}
	
	\begin{proof}Let $h\colon (-R,R)\to\R$ be a height zero solution as in \autoref{zeroheight}. We have the relation $\sigma_{2}\circ\Phi_s=\Phi_s\circ\sigma_{2}$, so that we can extend the surface by reflecting through $\{y=0\}$. This extension gives rise to a closed curve $\gamma$. The curve $\gamma$ is smooth since $h$ is asymptotic to a $y$-axis by \autoref{zeroheight} (c). Property \autoref{zeroheight} (b) of $h$ implies embeddedness of $\gamma$. This proves the claim about the generating curve.
		
		For the claimed invariances we note that invariance under $\sigma_{2}$ is obvious by construction of $\gamma$. The surface is invariant under $\sigma_{1}$ because of \autoref{lemode} (a). Similarly we argue for the invariance under $\rho_\pm$: Due to $R(a_0)=-a_0$, the initial values of the half-cylinder solutions remain invariant, hence we get the same solution. These isometries are all contained in $\Stab_{(0,0,0)}$ and generate it.\end{proof}
	
	\begin{Rem}[Uniqueness of height zero solution]\label{zeroheightremark}We conjecture there is exactly one height zero solution, but we do not have a proof at hand. If there were height zero solutions $h_0$ and $h_1$ to initial values $h_0(0)=a_0$ and $h_1(0)=a_1$ respectively, then both would satisfy $R(a_0)=-a_0\neq-a_1=R(a_1)$. Then one solution would be above the other one, so one cylinder would be on the mean convex side of the other cylinder. It appears we could use the maximum principle to rule out this situation. However, we cannot exhibit a point of tangential contact by moving one solution along the $y$-axis because left translations along the $y$-axis and left translations along $c$ do not commute. It seems we need a more elaborate application of the maximum principle -- a \emph{half-space theorem} -- to rule out that an \mcH-cylinder can be on the mean convex side of another \mcH-cylinder. In order to apply the general half-space theorem by \cite{Mazet13}, we have to verify two crucial assumptions: First, the parabolicity of our cylinders, that is, they must be conformal to a punctured plane. This assumption is satisfied in our case due to translational invariance. Second, there is an assumption on the mean curvature of equidistant surfaces to the given \mcH-cylinder. It appears difficult to verify and we do not know whether it holds or not.\end{Rem}
	
	\begin{Rem}[Computed example]\label{cylindercomputed}We used Mathematica to calculate the horizontal \mcH-cylinders. We have computed the ODE in Appendix A, see \autoref{odesol} on page \pageref{odesol}. We set $H=1$. Upon iteration we calculated for $a:=-0.642176$ that \[h(R(a))<10^{-7}\qquad\text{and}\qquad R(a)=-a=0.642176,\] as expected by \autoref{zeroheight}. Finally we extended the solution curve by a reflection through $\{y=0\}$. See \autoref{imodegraph}.
		
		\cite{Lopez14} also has a numerical example, but we believe it is less precise due to a different approach of exhibiting the initial value $h(0)=a$ numerically. For comparison, we note that $h(0)\approx-0.6425$ in \cite{Lopez14}, which we consider a less precise value. For instance, it does not satisfy $R(a)=-a$ numerically and we get ${h(R(a))\approx 2\cdot 10^{-4}}$. 
		
	\end{Rem}
	It is natural to look at the family of \mcH-cylinders with $H\in(0,\infty)$. Computations with Mathematica, illustrated by \autoref{cylinderevolution}, are evidence for the following:
	
	\begin{Conj}\label{emcylconj}The \mcH-cylinders with axis $c$ form an analytic family in ${H\in(0,\infty)}$. For $H\to0$ the surfaces are unbounded and for $H\to\infty$ they shrink to $c$.
	\end{Conj}
	
	
	\begin{figure}
		\begin{center}
			\includegraphics[width=1\linewidth]{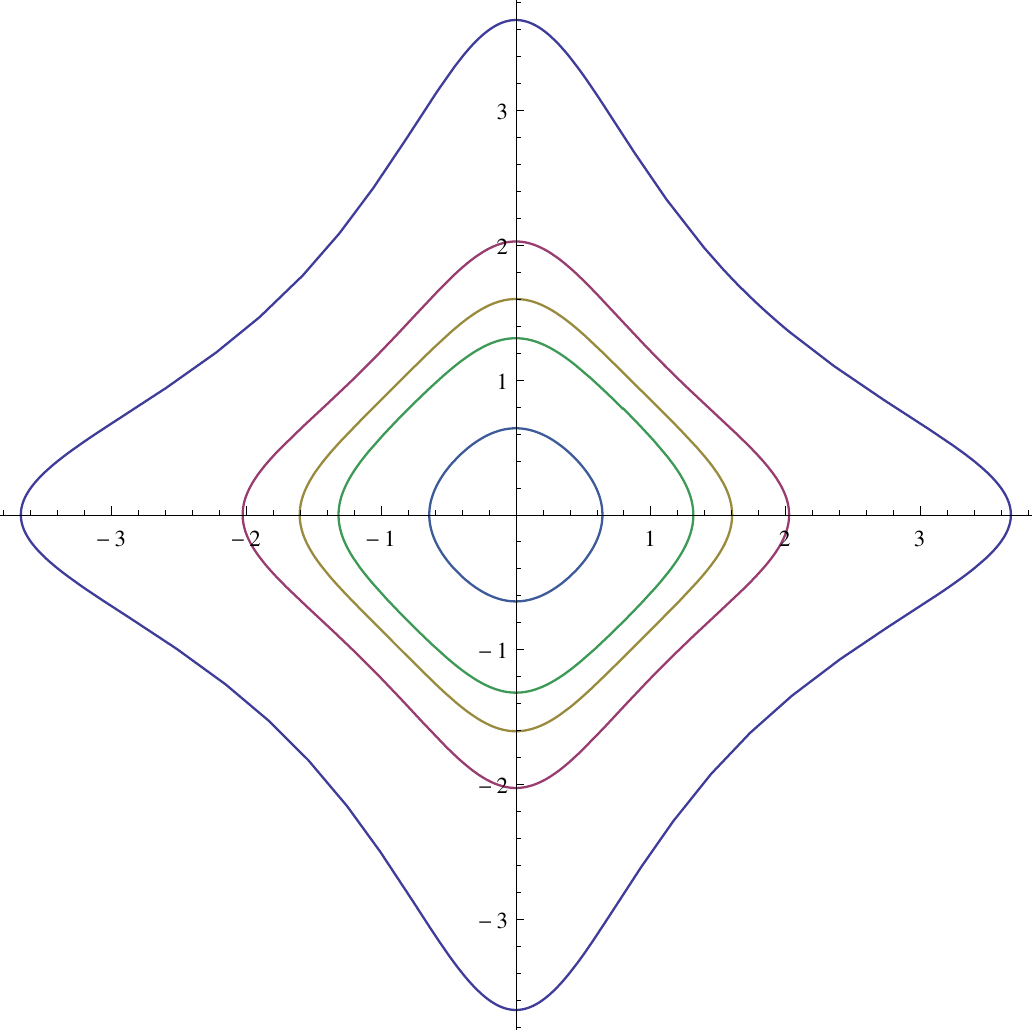}
		\end{center}
		
		\caption{Generating curves of the \mcH-cylinders of \autoref{thmcyl}: From outer to inner contour the mean curvature $H$ takes the values $0.5$, $0.6$, $0.65$, $0.7$ and $1$.}
		\label{cylinderevolution}
	\end{figure}
	
	\subsection{Conjecture on non-embedded solutions with axis \texorpdfstring{$c$}{c}}A shooting method leads to computed examples of non-embedded \mcH-cylinders with axis $c$ in $\Sol$. We shoot orthogonally from the vertical geodesic $c_+$, see \eqref{solgeodesics}, and aim at the $y$-axis, compare with \autoref{immersed}. Assume the solution curve $\gamma_0=(x_0,y_0,0)$ of the ODE \eqref{odecurve} meets the $y$-axis at $T>0$. Then $y_0^\prime(T)$ determines the angle between $\gamma_0$ and the $y$-axis. We extend this portion by the half-turn rotation $\rho_+$ about $c_+$ and reflections through $\{x=0\}$ and $\{y=0\}$, all of which leave the $(x,y)$-plane invariant, to a closed curve $\gamma=(x,y,0)$. The resulting curve is built up from $8$ such portions, possibly non-smooth at multiples of $T$.
	
	Recall that the turning number $\operatorname{turn}\left(\gamma\right)$ satisfies 
	\[
	2\pi\turn\left(\gamma\right)=\int_0^{8T}\kappa_{\operatorname{eucl}}\left(\gamma\right)\,dt+\operatorname{ext}\left(\gamma\right),
	\]
	where the second term $\operatorname{ext}\left(\gamma\right)=8y_0^\prime(T)$ denotes the sum of the exterior angles. If $\gamma_0$ meets the $y$-axis orthogonally at $T$ then $\gamma$ is smooth and $\ext(\gamma)=0$.
	
	
	To compute examples we fix $H=1$ and proceed as follows:
	\begin{itemize}
		\item Take $\gamma_0(0)=(d,d,0)$ for some $d\in\R$ and $\gamma_0^\prime(0)=\frac{1}{\sqrt{2}}(-1,1,0)$ as initial values.
		
		\item Suppose the resulting curve meets the $y$-axis at time $T=T(d)>0$.
		
		\item Vary $d$ while maintaining the same turning number of closed extension curve $\gamma$.
		
		\item Exhibit $d_1$ and $d_2$ with $y_0^\prime(T(d_1))<0$ and $y_0^\prime(T(d_2))>0$. An intermediate value argument gives some $d_0$ between $d_1$ and $d_2$ with $y_0^\prime(T(d_0))=0$.
	\end{itemize}

	
	
	With this ansatz we computed solutions with turning number $9$ and $17$, shown in \autoref{immersed1} and \autoref{immersed2}. 
	
	Aiming at the other vertical geodesic $c_-$ instead of the $y$-axis we find solutions with further turning numbers. See \autoref{immersed3} and \autoref{immersed4} for solutions with turning number $13$ and $21$.
	
	It is straightforward to compute more examples with turning number $5+4k$ where $k\in\N$. The particular value $d=0.429474$ corresponds to the solution generating the embedded horizontal cylinder. 
	
	Moreover, we computed an example with turning number $5$ for the values $H=\frac{1}{2}$ and $d=-0.965$. Increasing $H$ as well as $d$, we computed examples with turning number $5$ up to $H=0.759$. It appears that these solutions with turning number $5$ degenerate to the fivefold cover of a cylinder solution for some $H_0\in(0.759,1)$; see \autoref{bifurcation}.

	\begin{figure}
		\begin{minipage}[t]{0.495\textwidth}
			\begin{center}
				\includegraphics[width=1\textwidth]{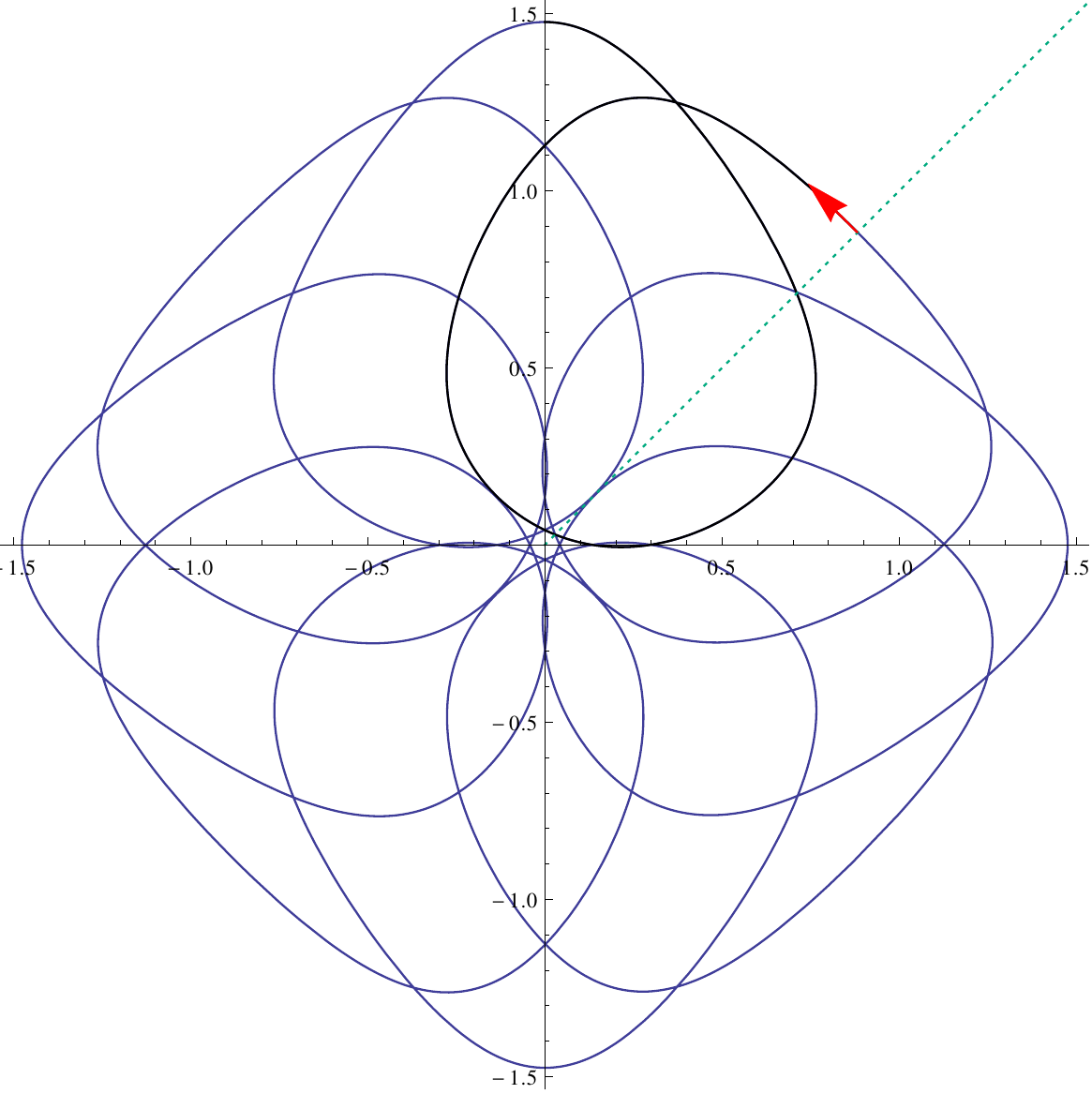}
			\end{center}
			\subcaption{$\turn(\gamma)=9$ and $d=0.8856$}\label{immersed1}
		\end{minipage}
		\begin{minipage}[t]{0.495\textwidth}
			\begin{center}\includegraphics[width=1\textwidth]{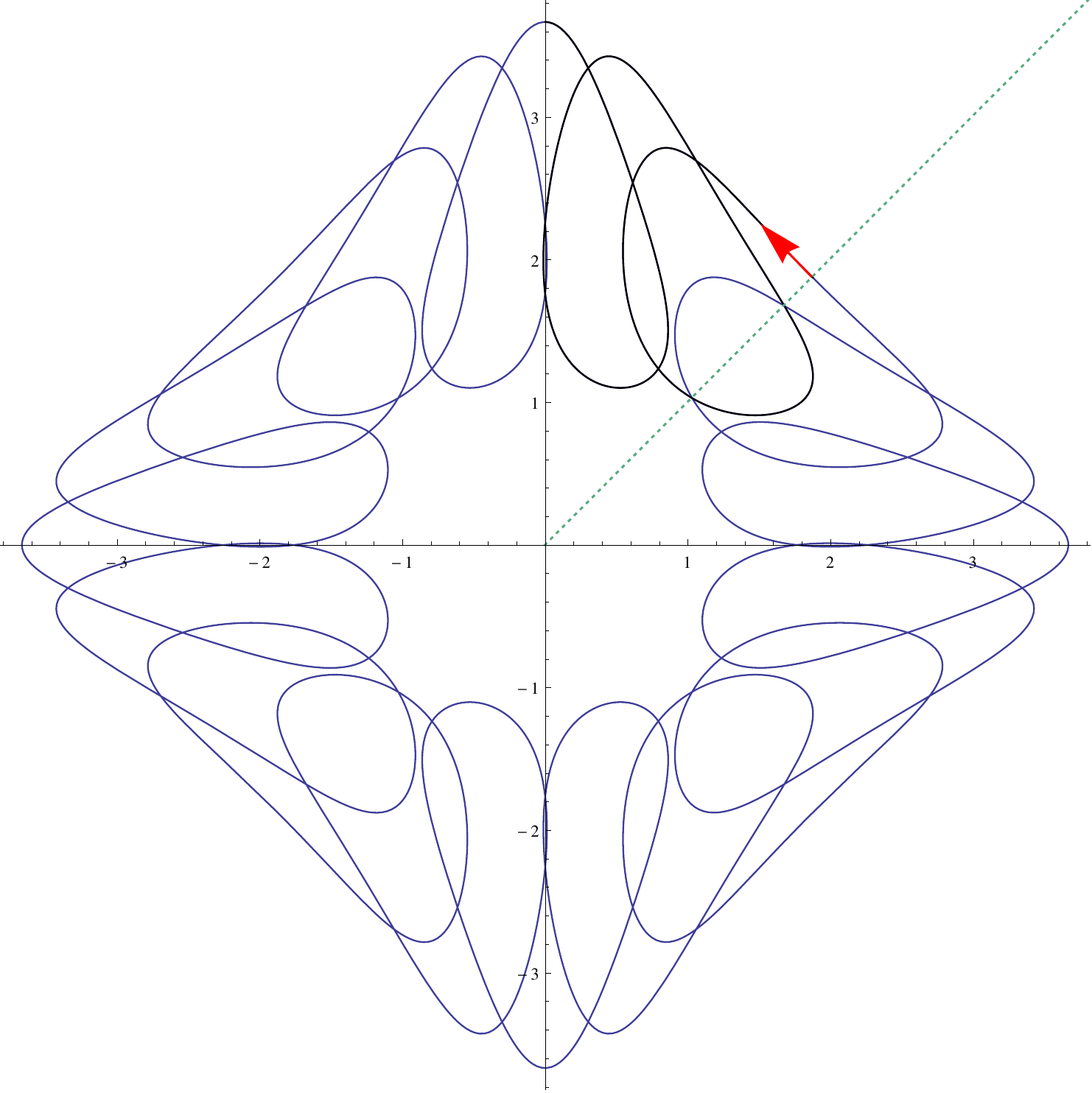}\end{center}
			\subcaption{$\turn(\gamma)=17$ and $d=1.8755$}\label{immersed2}
		\end{minipage}
		\begin{minipage}[b]{0.495\textwidth}
			\begin{center}\includegraphics[width=1\textwidth]{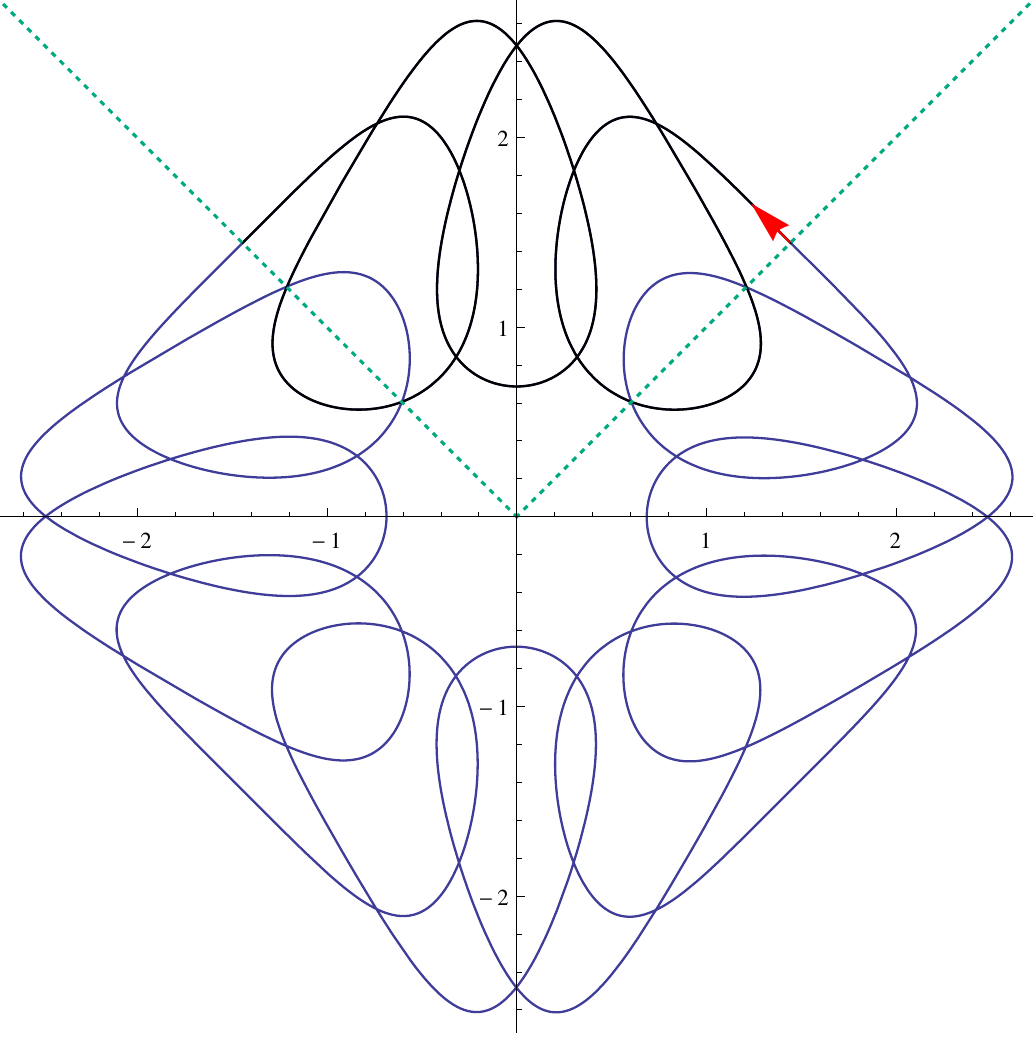}\end{center}
			\subcaption{$\turn(\gamma)=13$ and $d=1.445$}\label{immersed3}
		\end{minipage}
		\begin{minipage}[b]{0.495\textwidth}
			\begin{center}\includegraphics[width=1\textwidth]{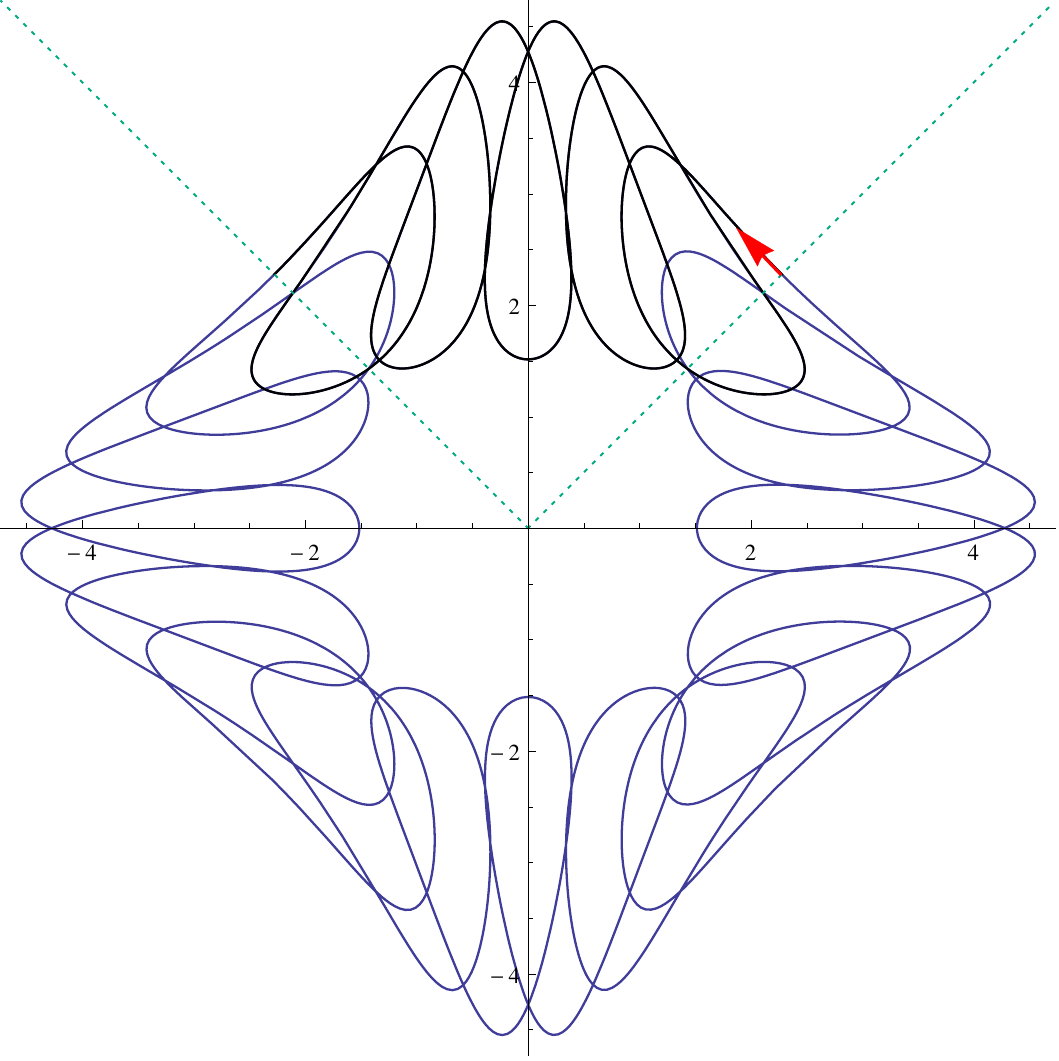}\end{center}
			\subcaption{$\turn(\gamma)=21$ and $d=2.277$}\label{immersed4}
		\end{minipage}
		\caption{Computed examples of solution curves $\gamma$ with turning number $5+4k$, where $k\in\{1,2,3,4\}$. They generate non-embedded cylinders with $H=1$. A fundamental portion of the curve is shown in black; it meets the dotted diagonals or the $y$-axis at right angle and generates the solution curve upon reflection.}
		\label{immersed}
	\end{figure}
	
	\begin{figure}
		\begin{minipage}[t]{0.495\textwidth}
			\begin{center}
				\includegraphics[width=1\textwidth]{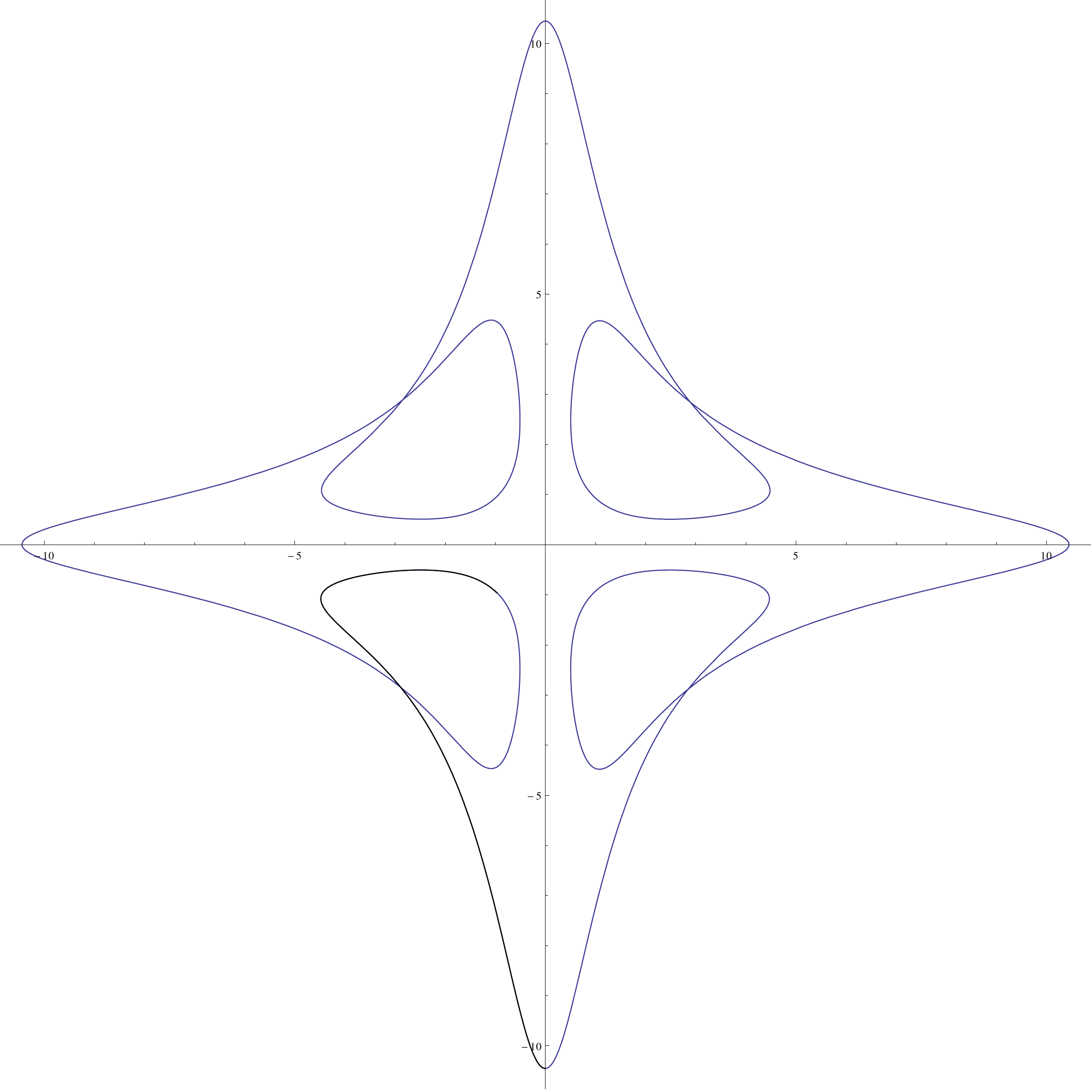}
			\end{center}
			\subcaption{$H=0.5$ and $d=-0.965$}
		\end{minipage}
		\begin{minipage}[t]{0.495\textwidth}
			\begin{center}\includegraphics[width=1\textwidth]{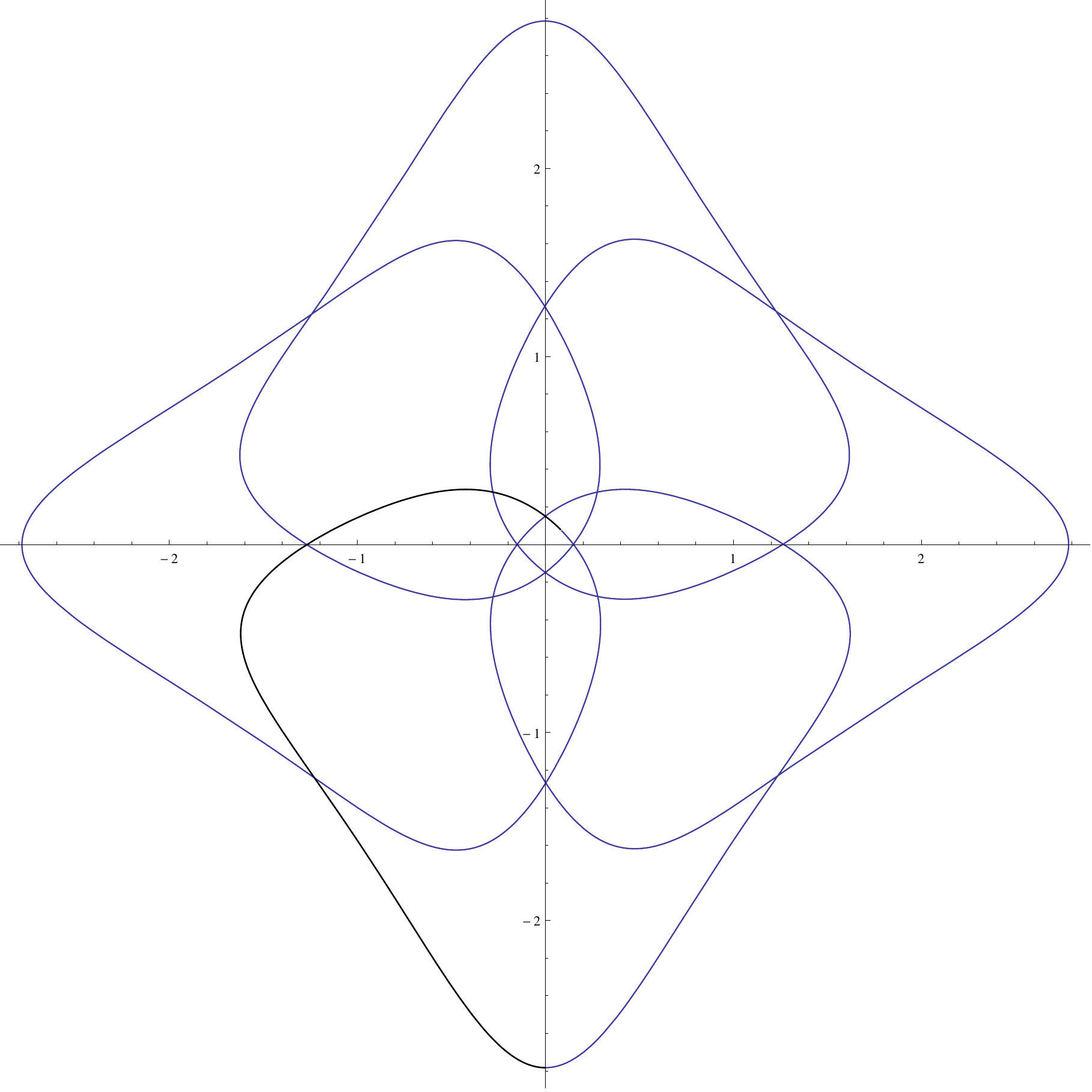}\end{center}
			\subcaption{$H=0.7$ and $d=0.08$}
		\end{minipage}
		\begin{minipage}[b]{0.495\textwidth}
			\begin{center}\includegraphics[width=1\textwidth]{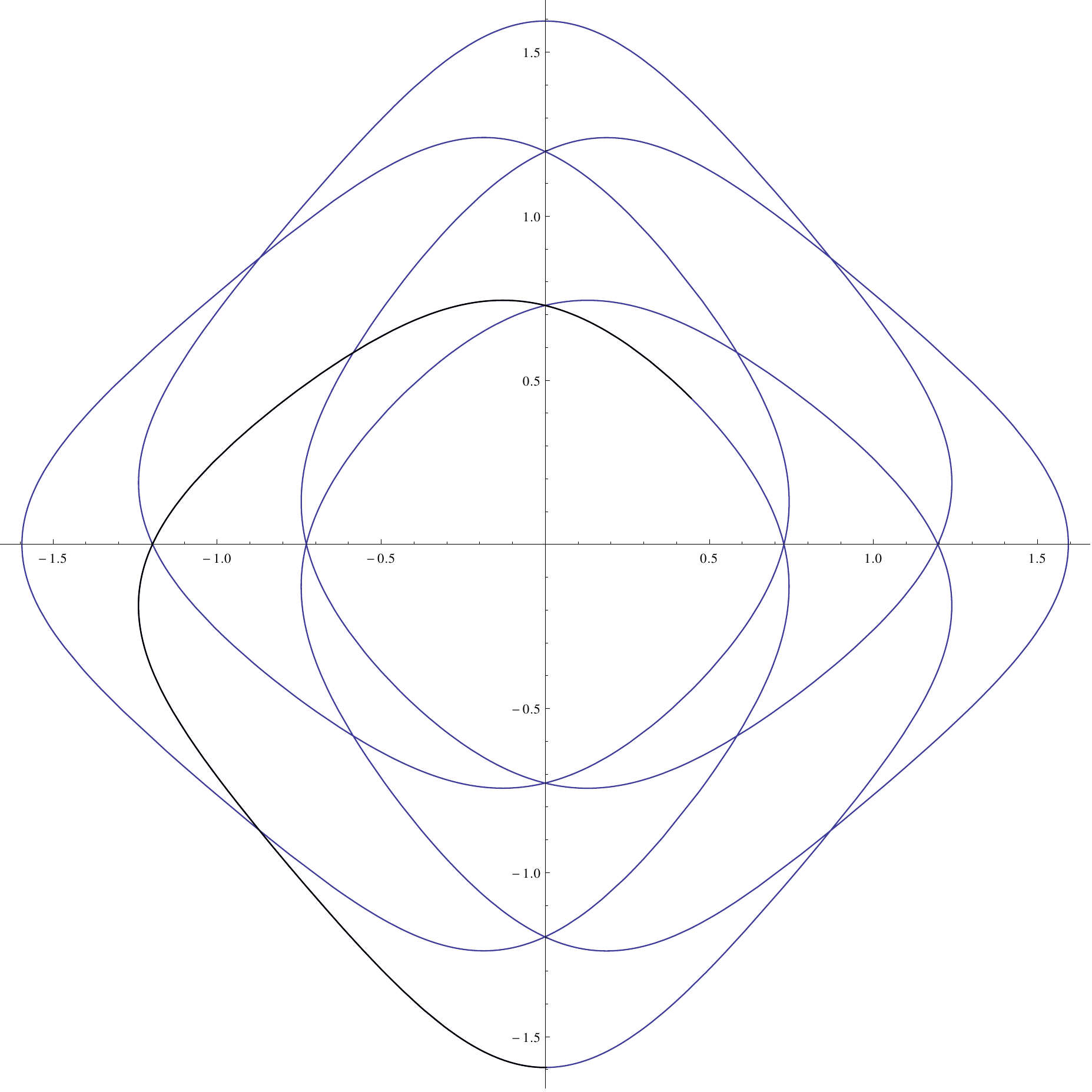}\end{center}
			\subcaption{$H=0.75$ and $d=0.445$}
		\end{minipage}
		\begin{minipage}[b]{0.495\textwidth}
			\begin{center}\includegraphics[width=1\textwidth]{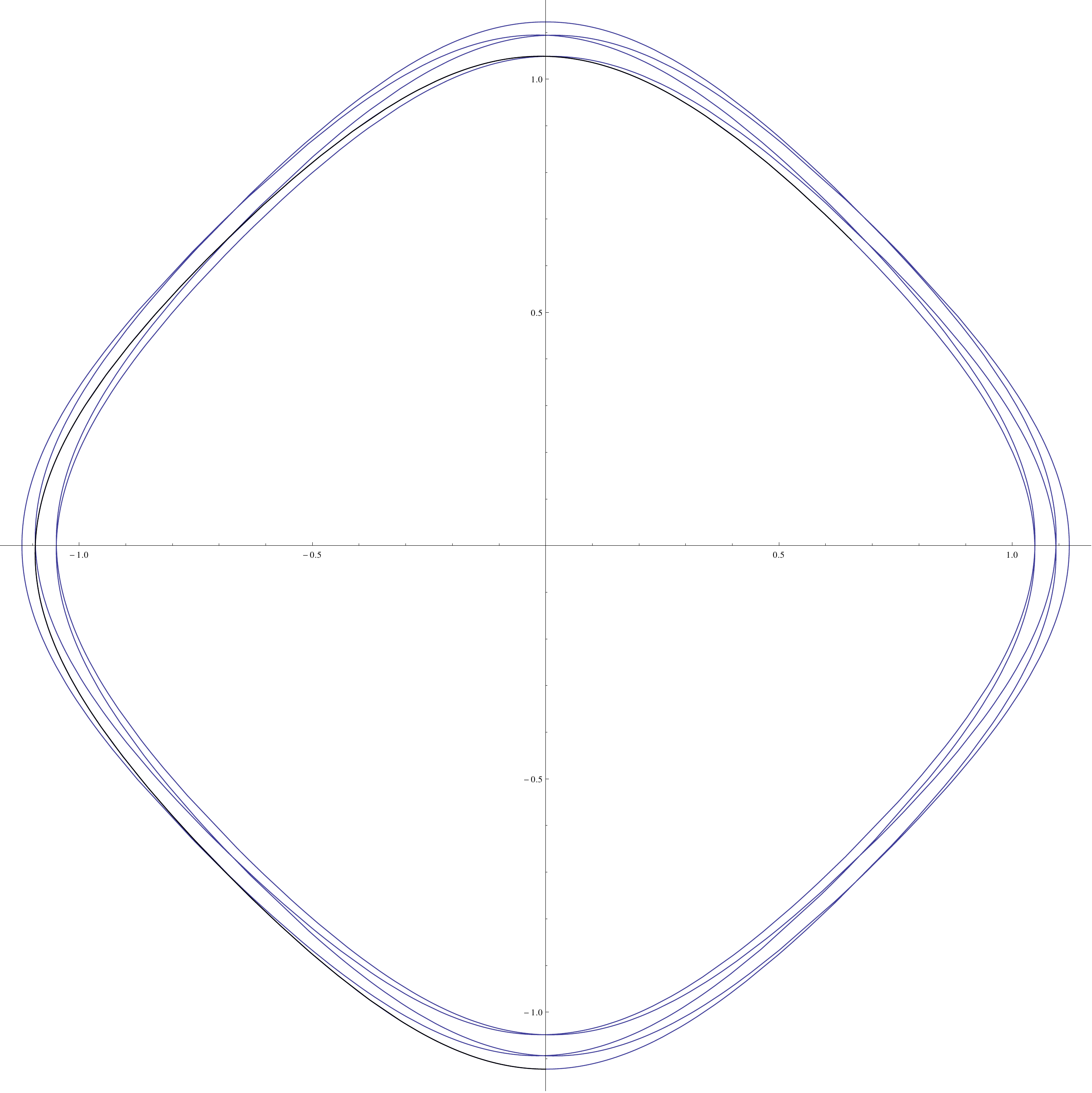}\end{center}
			\subcaption{$H=0.759$ and $d=0.655$}
		\end{minipage}
		\caption{Solution curves with turning number $5$ converge to a multiple cover of the embedded \mcH-cylinder solution upon increasing $H$ and $d$.}\label{bifurcation}
	\end{figure}
	
	\begin{Conj}\label{nonembeddedcyl}For each $H>0$ there is $m=m(H)\in\N$ such that for every natural number $k\geq1$ there exists a non-embedded closed curve with turning number $m+4k$ as generating curve of a $(\Phi_s)_{s\in\R}$-invariant surface with constant mean curvature $H$.
	\end{Conj}
	A proof of this conjecture seems beyond the techniques used in the present paper.
	
	
	\section{Vertical \texorpdfstring{\mcH}{mcH}-cylinders in \texorpdfstring{$\Sol$}{Sol}}\label{chapter13}
	We now study constant mean curvature surfaces invariant under $(\Psi_s^\pm)_{s\in\R}$, which are left translations along $c_\pm$; for the definitions see \eqref{solgeodesics} and \eqref{cpmtranslation}. We proceed as in the previous section for surfaces invariant under $(\Phi_s)_{s\in\R}$. Since $c_\pm$ is a vertical geodesic in $\Sol$, we call these surfaces \emph{vertical \mcH-cylinders}.

	\subsection{ODE for surfaces invariant by translations along \texorpdfstring{$c_\pm$}{c+-}}For our second surface family we can consider the foliation $(S_{s}^\pm)_{s\in\R}$ of planes above the vertical geodesics \[\{(x,\pm x,0)\colon x\in\R\}\subseteq\{z=0\}\mbox{.}\] We have
	\[
	S_{s}^\pm=\left\{\Psi_s^\pm(x,\mp x,z)\colon x,z\in\R\right\}=\left\{(x+s,\mp x\pm s,z)\colon x,z\in\R\right\},\quad s\in\R.
	\]
	A surface invariant under $(\Psi_s^\pm)_{s\in\R}$ can then be defined as follows: For $\mathcal{C}^2$-functions $x,z\colon J\to\R$, defined on an open interval $J\subset\R$, the curve
	\[\gamma\colon J\to\Sol,\qquad \gamma(t):=(x(t),\mp x(t),z(t))\]
	is in $S_0^\pm$ and the invariant surface generated by left translation of $\gamma$ along $c_\pm$ is paramet\-rised by
	\begin{equation}
	f\colon\R\times J\to\Sol,\qquad f(s,t):=\Psi_s^\pm(\gamma(t))=(x(t)+s,\mp x(t)\pm s,z(t)).\label{surfacecpm}
	\end{equation}
	Again, we consider graphical solutions. In the previous family we used $y$-graphs, which are also Killing graphs with respect to the Killing field $K_2=\partial_y$. Here, we study $z$-graphs with respect to the fibre projection \begin{equation}\mathcal{F}\colon\Sol\to\R^2\times\{0\}\mbox{,}\qquad \mathcal{F}(x,y,z):=(x,y,0).\label{fibreprojection}\end{equation}
	A discussion as in \autoref{odesol2} gives the following result for the ODE of graphical solutions:\index{invariant surface in $\Sol$}
	\begin{Prop}\label{odesol22}Let $H\in\R$. There is a smooth function $F\colon\R^3\to\R$ such that the invariant surface
		\[
		f\colon\R\times J\to\Sol,\qquad f(s,t):=\Psi_{s}^\pm\left(t,\mp t,h(t)\right)\qquad\text{where }h\in\mathcal{C}^2(J,\R),
		\]
		has constant mean curvature $H$ with respect to the inner normal if and only if 
		\begin{equation}h^{\prime\prime}(t)=F(t,h(t),h^\prime(t))\qquad \text{for all t }\in J.\label{odeprop2}\end{equation}
		%
	\end{Prop}
	
	\begin{proof}As in \autoref{odesol2} (a) we set $v_1:=\partial_sf$ and $v_2:=\partial_tf$. The inner normal to $f$ is denoted by $N$ and the coefficients of the first and second fundamental form are defined as $g_{ij}:=\langle v_i,v_j\rangle$ and $b_{ij}:=\langle \nabla_{v_i}v_j,N\rangle$ for $i,j\in\{1,2\}$, respectively. The mean curvature of $f$ is then given by
		\[
		H=\frac{b_{11}g^{11}+2b_{12}g^{12}+b_{22}g^{22}}{2}.
		\]
		Noting
		\[f(s,t)=(t+s,\mp t\pm s,h(t)),\]
		and in view of \eqref{killing}, we obtain \begin{align*}v_1&=K_1\pm K_2=e^{h}E_1\pm e^{-h}E_2,\\
		v_2&= K_1\mp K_2+h^\prime E_3=e^{h}E_1\mp e^{-h}E_2+h^\prime E_3.\end{align*}
		Thus $H$ only depends on $t$, $h(t)$, $h^\prime(t)$ and $h^{\prime\prime}(t)$.
		We get
		\begin{align*}
		\nabla_{v_2}v_2&=\nabla_{v_2}e^{h}E_1\mp \nabla_{v_2}e^{-h}E_2+\nabla_{v_2}h^\prime E_3\\
		&=\underbrace{h^\prime e^{h}E_1+e^{h}\nabla_{v_2}E_1\pm h^\prime e^{h}E_2\mp e^{-h}+\nabla_{v_2}E_2h^\prime \nabla_{v_2}E_3}_{=:w}+h^{\prime\prime}E_3
		\end{align*}
		It is obvious that $w$ does not depend on $h^{\prime\prime}$ and the only term of $H$ containing $h^{\prime\prime}$ is
		\[
		\frac{b_{22}g^{22}}{2}=\frac{b_{22}g_{11}}{2\det(g)}=\frac{\langle \nabla_{v_2}v_2,N\rangle g_{11}}{2\det(g)}=\frac{(h^{\prime\prime}\langle N,E_3\rangle +\langle w,N\rangle)g_{11}}{2\det(g)}.
		\]
		Since we have a $z$-graph at hand and $N$ is the inner normal, we see $\langle N,E_3\rangle>0$. Moreover $g_{11}>0$ because $\Psi_{s}^\pm$ never acts trivially. The arguments for the existence of a smooth function $F\colon\R^3\to\R$ with $h^{\prime\prime}(t)=F(t,h(t),h^\prime(t))$ are now as in \autoref{odesol2} (a).\end{proof}

	\subsection{Geometric discussion of ODE and its extension \texorpdfstring{\mcH}{mcH}-cylinders with axis \texorpdfstring{$c_\pm$}{cplusminus}}We discuss the ODE \eqref{odeprop2} now. Instead of formulating separate versions of \autoref{lemode} and \autoref{zeroheight}, we will point out the differences for this particular geometric setting when proving the following result:

	
	\begin{Theorem}\label{thmcyl2}For each $H>0$ there is a smooth embedded simple closed curve $\gamma$ in $S_{0}^\pm=\left\{\left(x,\mp x,z\right)\colon x,z\in\R\right\}$ which generates a $(\Psi_s^\pm)_{s\in\R}$-invariant embedded surface $f(s,t)=\Psi_s^\pm(\gamma(t))$ in $\Sol$ with constant mean curvature $H$, invariant by the half-turn rotations $\rho_+$ and $\rho_-$.
	\end{Theorem}
	We call these surfaces \emph{\mcH-cylinders with axes $c_+$ and $c_-$} or, since $c_+$ and $c_-$ are vertical geodesics, also \emph{vertical \mcH-cylinders in $\Sol$}. We refer to \eqref{solgeodesics}, \eqref{cpmtranslation} and \eqref{solrotations} for the definitions of $c_\pm$, $\Psi_s^\pm$ and $\rho_\pm$, respectively.\index{cylinder in $\Sol$}
	
	\begin{proof}The discussion from \autoref{chapter122} is also applicable to surfaces invariant by left translation $(\Psi_s^\pm)_{s\in\R}$ along $c_\pm$. Therefore we only indicate the differences in the proofs of \autoref{lemode} and \autoref{zeroheight}.
		
		First we fix some notation. We let $h\colon I_{\max}\to\R$ be the unique maximal solution of \eqref{odeprop2} with $h(0)=a$ and $h^\prime(0)=0$. The surface $f$ is given by
		\[f(s,t)=\Psi_{s}^\pm(t,\mp t,h(t))=(t+s,\mp t\pm s,h(t)),\]
		which is the surface invariant under the left translations $(\Psi_s^\pm)_{s\in\R}$ and generated by the graph $t\mapsto(t,\mp t,h(t))$. We denote its image by $\Sigma$. We will use frequently that left translations along $c_\mp$ and $z$-axes are isometries. We also make use of \autoref{solsphere}.
		
		\emph{Symmetric solution with bounded existence interval and graph of finite height:} 
		
		In order to obtain a symmetric solution as in \autoref{lemode} (a), we fix the initial values at $h(0)=a$ and $h^\prime(0)=0$ and argue as follows: The half-turn rotation $\rho$ about $c$ satisfies
		\begin{equation}\Psi_{-s}^\pm\circ\rho=\rho\circ\Psi_s^\pm\label{commutatorsol2}
		\end{equation}
		This follows directly from \eqref{solrotations} and \eqref{cpmtranslation}, because all these mappings are also Euclidean isometries. The surfaces $f$ and $\tilde{f}:=\rho\circ f$ are isometric surfaces, both invariant under $(\Psi_s^\pm)_{s\in\R}$ due to \eqref{commutatorsol2}. However, $\tilde{f}$ is generated by $(-t,\pm t,h(t))$. Since both of these graphs have the same initial values, they must be equal and we obtain a symmetric solution, that is $I_{\max}=(-R,R)$ for some $R=R(a)\in(0,\infty]$ and $h(t)=h(-t)$ for all $t\in(-R,R)$.
		
		The arguments for the bounded existence interval carry over when replacing $\Pi_y$, defined in the proof of \autoref{lemode} (b), with the fibre projection $\mathcal{F}$ as defined in \eqref{fibreprojection}. We then move \mcH-spheres by left translation along some $z$-axis towards $\Sigma$. The graph $h$ satisfies $\lim_{t\to\pm R}\lvert h(t)\rvert<\infty$ for if it were false we could move \mcH-spheres towards $\Sigma$ by left translation along $c_\mp$. Thus we also have a bounded graph in this case.
		
		\emph{Height zero half-cylinder solution:} The arguments from \autoref{zeroheight} (a) and (b) can be copied to show existence of $a_0<0$ such that the solution $h\colon(-R(a_0),R(a_0))\to\R$ of \eqref{odeprop2} with initial values $h(0)=a_0$ and $h^\prime(0)=0$ satisfies $h(\pm R(a_0))=0$ and $h(t)<0$ for all $t\in(-R(a_0), R(a_0))$. When copying the arguments, one has to replace $y$-axes by $z$-axes to move \mcH-spheres along. The asymptotic behaviour at the boundary follows in exactly the same way, that is, we have $\lim_{t\to\pm R}h^\prime(t)=\pm\infty$.
		
		\emph{Existence of vertical cylinders:} In the final step we take such height zero solution. Since $\rho_\pm$ commutes with left translation along $c_\pm$, see \eqref{cpmtranslation} and \eqref{solrotations}, we can extend this solution to a smooth embedded closed curve by the half-turn rotation $\rho_\pm$ about the axis $c_\pm$. The half-turn rotation $\rho_\mp$ satisfies $\rho_\mp\circ\Psi_s^\pm=\Psi_{-s}^\pm\circ\rho_\mp$ and thus leaves the surface invariant. This finishes the proof.\end{proof}
	
	\begin{Rem}Constant mean curvature surfaces invariant under $(\Psi_s^\pm)_{s\in\R}$ in $\Sol$ have not been considered before and this result shows that interesting surfaces are generated. We have not computed the ODE for these surfaces so that we do not make any claims about non-embedded solutions with axis $c_\pm$.\end{Rem}
	
	We believe \autoref{emcylconj} also applies to embedded \mcH-cylinders with axis $c_\pm$:
	\begin{Conj}\label{emcylconj2}The \mcH-cylinders with axis $c_\pm$ form an analytic family with respect to  ${H\in(0,\infty)}$. For $H\to0$ the surfaces are unbounded and for $H\to\infty$ they shrink to $c_\pm$.
	\end{Conj}
	
	\part{\texorpdfstring{\mcH}{mcH}-cylinders in \texorpdfstring{$E(\kappa,\tau)$}{E(k,t)}-spaces diffeomorphic to \texorpdfstring{$\R^3$}{R3}}
	The $E(\kappa,\tau)$-spaces are Riemannian fibrations $E\to B$ with geodesic fibres, bundle curvature $\tau\in\R$ and base curvature $\kappa\in\R$. We will only consider those diffeomorphic to $\R^3$. These arise for $\kappa\leq0$, i.e., we exclude the Berger spheres and $\mathbb{S}^2\times\R$. In Section 4 we describe these spaces. Most results concerning constant mean curvature surfaces then become ``horizontal`` or ``vertical`` generalisations of results in $\R^3$. It turns out that the arguments given in Section 2 and 3 carry over to prove existence of tilted \mcH-cylinders in $E(\kappa,0)$ and of horizontal \mcH-cylinders in $E(\kappa,\tau)$ for $\tau\neq0$. In the final Section we compute the horizontal diameter of a horizontal \mcH-cylinder in $E(\kappa,\tau)$-spaces with $\kappa\leq0$.
	
	\section{Preliminaries on \texorpdfstring{$E(\kappa,\tau)$}{E(k,t)}-spaces}\label{chapter21}
	
	First we introduce some general notations for and properties of $E(\kappa,\tau)$-spaces. Then we describe an explicit model for $\kappa\leq0$ and only work therein. This model is a metric Lie group and we specify geodesics, left translations along them and other properties needed to carry over the arguments from Section 2.
	\subsection{General notation and properties}\label{sec41}The $E(\kappa,\tau)$-spaces are simply connected homogeneous three-manifolds $E$ diffeomorphic to $\R^3$, $\mathbb{S}^3$ or $\mathbb{S}^2\times\R$ and arise as Riemannian fibrations $\Pi\colon E\to B$ with geodesic fibres, where $B$ has curvature $\kappa\in\R$ and the bundle curvature\index{bundle curvature} is $\tau\in\R$. They have some geometric properties, which can be stated without an explicit model.  
	
	\subsubsection{Slope of geodesics}\label{slopesection}Associated to each $E(\kappa,\tau)$-space with Riemannian submersion $\Pi\colon E\to B$ is a Killing field $\xi$, which is tangent to the geodesic fibres. As a consequence of \emph{Clairaut's Theorem}, geodesics have the following property:
	\begin{Prop}[{{\cite[Lemma 3.7]{Engel}}}]\label{slope}
		Let $c\colon\R\to E$ be a unit-speed geodesic in an $E(\kappa,\tau)$-space with Riemannian submersion $\Pi\colon E\to B$. Then there is $\alpha\in[0,\pi]$ with $\langle c^\prime,\xi\circ c\rangle\equiv\cos(\alpha)$. We call $\alpha$ \emph{slope of $c$ with respect to $\xi$}.
		The projection $\tilde{c}:=\Pi\circ c$ is a curve of constant geodesic curvature $-\tau\cot(\alpha)$ in $B$.
	\end{Prop}
	
	We call the geodesic fibres, corresponding to $\alpha=0$ and $\alpha=\pi$, \emph{vertical geodesics}\index{vertical geodesic}.  On the other hand, the case of $\alpha=\frac{\pi}{2}$ corresponds to \emph{horizontal geodesics}. We refer to the other cases as \emph{tilted geodesics}.
	
	For $\tau=0$ \emph{all} geodesics project to geodesics of the base space $B$, while for $\tau\neq0$ only horizontal geodesics project to geodesics of $B$ (evident from \autoref{slope}). In the following we are only considering geodesics $c$ in $E$ which project onto geodesics in $B$ and have slope $\alpha\in(0,\pi)$.
	
	\subsubsection{Vertical planes}A \emph{vertical plane} $P$ is the preimage $P=\Pi^{-1}(\tilde{c})\subseteq E$ where $\tilde{c}\colon \R\to B$ is a geodesic in the base $B$. A vertical plane is totally geodesic if and only if $\tau=0$. See for example \cite{Notes2009}. 
	
	\subsubsection{Isometries induced by vertical and horizontal geodesics/planes}
	In the $E(\kappa,\tau)$-spaces we have the following isometries:
	\begin{Prop}\label{isometries}Let $E$ be an $E(\kappa,\tau)$-space with Riemannian submersion $\Pi\colon E\to B$ and vertical Killing field $\xi$.
		\begin{enumerate}[(a)]
			\item The flow of $\xi$ is a one-parameter family of isometries $(T_s)_{s\in\R}$, called \emph{vertical translations}.
			
			\item Rotations about fibres, that is, about vertical geodesics, are isometries.
			
			\item Each horizontal geodesic admits a half-turn rotation, i.e., an isometric rotation of angle $\pi$ about the horizontal geodesic.
			
			\item For $\tau=0$, reflections through vertical planes are isometries.
		\end{enumerate} 
	\end{Prop}
	These properties are standard and can be found in \cite{Notes2009}.
	

	\subsection{\texorpdfstring{$E(\kappa,\tau)$}{E(k,t)}-spaces with \texorpdfstring{$\kappa\leq0$}{k<=0}}\label{sec42}
	The $E(\kappa,\tau)$-spaces with $\kappa\leq0$ are diffeomorphic to $\R^3$ and arise as metric Lie groups. We describe a model and specify some geometric properties. The advantage of this model is that the limits $\kappa\to0$ and $\tau\to0$ are well-defined, also on the level of orthonormal frames.
	
	\subsubsection{Model}For our purpose the classification of \cite{MP} provides a convenient description of these spaces. For $\kappa\leq 0$ and $\tau\in\R$ let \[A(\kappa,\tau):=\begin{pmatrix}\sqrt{-\kappa}&0\\2\tau&0\end{pmatrix}.\]
	We want to compute
	\[\begin{pmatrix}a_{11}(z)&a_{12}(z)\\ a_{21}(z)&a_{22}(z)\end{pmatrix}:=e^{zA(\kappa,\tau)}.\]
	For $\kappa<0$ we have
	\[e^{zA(\kappa,\tau)}=\begin{pmatrix}e^{z\sqrt{-\kappa}}&0\\ \frac{2\tau}{\sqrt{-\kappa}}\left(e^{z\sqrt{-\kappa}}-1\right)&1\end{pmatrix}\]
	and for $\kappa=0$ we get
	\[
	e^{zA(0,\tau)}=\begin{pmatrix}1&0\\2\tau z&0\end{pmatrix}.
	\]
	We observe $\lim_{\kappa\to0}e^{zA(\kappa,\tau)}=e^{zA(0,\tau)}$ for all $z,\tau\in\R$ so that the first expression also makes sense for $\kappa=0$.
	
	The space $\R^2\ltimes_{A(\kappa,\tau)}\R$ is a metric Lie group with group structure\index{left-translation}
	\begin{equation}
	(x_1,y_1,z_1)\ast (x_2,y_2,z_2):=\left((x_1,y_1)+e^{z_1A(\kappa,\tau)}(x_2,y_2),z_1+z_2\right)\label{ektgroup}
	\end{equation}
	and Riemannian metric
	\begin{equation}\begin{split}\langle\cdot,\cdot\rangle_{(x,y,z)}&=\left(e^{-2z\sqrt{-\kappa}}-\frac{4\tau}{\kappa}(e^{-z\sqrt{-\kappa}}-1)^2\right)\,dx^2+dy^2+dz^2\\
	&+\frac{2\tau}{\sqrt{-\kappa}}(e^{-z\sqrt{-\kappa}}-1)(dx\otimes dy+dy\otimes dx).\end{split}\label{metricekt}
	\end{equation}
	We also write $\Ekt:=(\R^3,\langle\cdot,\cdot\rangle)$. Left multiplication by $p\in\Ekt$ is therefore an isometry \(\Lcal_p\colon \Ekt\to \Ekt\mbox{, }\Lcal_p(g):=p\ast g\mbox{.}\)
	The canonical orthonormal frame, obtained by left translation of the Euclidean frame from the origin $(0,0,0)$, is
	\begin{align}E_1(x,y,z)&=e^{z\sqrt{-\kappa}}\partial_x+\frac{2\tau}{\sqrt{-\kappa}}\left(e^{z\sqrt{-\kappa}}-1\right)\partial_y,\nonumber\\
	E_2(x,y,z)&=\partial_y,\label{ektframe}\\
	E_3(x,y,z)&=\partial_z.\nonumber
	\end{align}
	The Riemannian connection with respect to this frame has the following representation:\begin{equation}\begin{array}{lll}\nabla_{E_1}E_1=\sqrt{-\kappa}E_3,&\nabla_{E_1}E_2=\tau E_3,&\nabla_{E_1}E_3=-\sqrt{-\kappa}E_1-\tau E_2,\\
	\nabla_{E_2}E_1=\tau E_3,& \nabla_{E_2}E_2=0,&\nabla_{E_2}E_3=-\tau E_1,\\
	\nabla_{E_3}E_1=\tau E_2,&\nabla_{E_3}E_2=-\tau E_1,& \nabla_{E_3}E_3=0.
	\end{array}\label{ektconnection}\end{equation}
	
	We check that $\Ekt$ is indeed an $E(\kappa,\tau)$-space.
	\begin{Prop}\label{ektfibration}Let $\kappa\leq0$ and $\tau\in\R$. On $\R^2$ we consider the Riemannian metric
		\begin{equation}
		\tilde{g}_{(x,z)}=e^{-2z\sqrt{-\kappa}}dx^2+dz^2.\label{ektbase}
		\end{equation}
		Then $\Pi\colon(\R^3,\langle\cdot,\cdot\rangle)\to(\R^2,\tilde{g})\mbox{, }(x,y,z)\mapsto(x,z)$ is a Riemannian submersion with geodesic fibres over the simply connected surface $(\R^2,\tilde{g})$ with constant curvature $\kappa$. This submersion has bundle curvature $\tau$, so that $(\R^3,\langle\cdot,\cdot\rangle)$ is isometric to $E(\kappa,\tau)$. The vertical Killing field is $\xi=E_2$ and its flow by vertical translations $(T_s)_{s\in\R}$ is given by\begin{equation}T_s(x,y,z)=(x,y+s,z)=\Lcal_{(0,s,0)}(x,y,z),\qquad s\in\R.\label{ektvtranslation}\end{equation}
		%
		%
	\end{Prop}
	
	\begin{proof}[Sketch of proof]We can refer to various Theorems in \cite{MP}, but let us give the explicit argument:
		\begin{itemize}
			\item The vertical space is spanned by $E_2$ while the horizontal space is spanned by $E_1$ and $E_3$.
			\item For a horizontal vector $v=\lambda E_1+\mu E_3$ we have $\tilde{g}_{(x,z)}(d\Pi\,v,d\Pi\,v)=\lambda^2+\mu^2$, so that $\Pi$ is indeed a Riemannian submersion.
			\item In view of the Riemannian connection we have $\langle R(E_1,E_3)E_3,E_1)=\kappa-3\tau^2$, so that $(\R^2,\tilde{g})$ is a simply connected surface with constant curvature $\kappa$.
			\item We also have $\frac{1}{2}\langle \nabla_{E_3}E_1-\nabla_{E_1}E_3,E_2\rangle=\tau$, which proves the claim about the bundle curvature.\index{bundle curvature}
			\item In order to show that $\xi=E_2$ is the Killing field associated to $\Pi$, it is sufficient to check $T_s(x,y,z)=\Lcal_{(0,s,0)}(x,y,z)$. Thus $T_s$ is an isometry, so that the infinitesimal generator $E_2=\partial_y$ of $(T_s)_{s\in\R}$ is indeed a Killing field.
			\qedhere
		\end{itemize}
	\end{proof}
	
	\subsubsection{Geodesics whose projection is a geodesic and left translations}
	As explained in \autoref{slopesection}, we only consider geodesics $c$ in an $E(\kappa,\tau)$-space whose projection in $B$ is also a geodesic. For the specific model $\Ekt$ with Riemannian submersion $\Pi(x,y,z)=(x,z)$, it is easy to verify that \[\R\to\Ekt,\qquad s\mapsto(0,0,s)\] is a horizontal geodesic. Indeed, its tangent vector is $E_3=\partial_z$ and $\nabla_{E_3}E_3=0$ by \eqref{ektconnection}. Since $\Ekt$ is homogeneous and rotations about fibres are isometries, it is sufficient to consider geodesics $c\colon\R\to \Ekt$ with $\Pi(c(s))=(0,s)$ and $c(0)=(0,0,0)$. These are given by
	\begin{equation}
	c\colon \R\to \Ekt,\; c(s):=\begin{cases}
	(0,0,s)&\mbox{if }\tau\neq0,\\
	(0,\cos(\alpha)\cdot s,\sin(\alpha)\cdot s)&\mbox{if }\tau=0,
	\end{cases}\label{ektgeodesic}
	\end{equation}
	where $\alpha\in(0,\pi)$ is the slope of $c$. See also \autoref{slope}. For $\tau\neq0$ we have shown why $c$ is geodesic. For $\tau=0$ we note $c^\prime=\cos(\alpha)E_2+\sin(\alpha)E_3$ and $\nabla_{E_j}E_k=0$ by \eqref{ektconnection} for all $j,k\in\{2,3\}$. This shows $\nabla_{c^\prime}c^\prime=0$.
	
	Our model is a metric Lie group and so a one-parameter family of isometries is $(\Phi_s)_{s\in\R}$ with $\Phi_s:=\Lcal_{c(s)}$. It satisfies $\Phi_s(0,0,0)=c(s)$ and we call it \emph{left translation along $c$}.
	
	For $\tau\neq0$ it is given by
	\begin{align}
	\Phi_s(x,y,z)=\Lcal_{(0,0,s)}(x,y,z)&=\left(e^{s\sqrt{-k}}x,2\tau\frac{e^{s\sqrt{-k}}-1}{\sqrt{-\kappa}}x+y,z+s\right)\nonumber\\
	&\overset{\kappa\to0}{\to}\left(x,2\tau sx+y,z+s\right),\label{translationekt}
	\end{align}
	and the infinitesimal generator or Killing field of $(\Phi_s)_{s\in\R}$ at $(x,y,z)\in \Ekt$ is 
	\begin{equation}
	K_{(x,y,z)}=\frac{d}{ds}\Big\vert_{s=0}\Phi_s(x,y,z)=x\sqrt{-\kappa}\,e^{-z\sqrt{-\kappa}}E_1+2\tau xe^{-z\sqrt{-\kappa}}E_2+E_3.\label{killingekt}
	\end{equation}
	For $\tau=0$ we have
	\begin{align}
	\Phi_s(x,y,z)&=\Lcal_{(0,s\cdot\cos(\alpha),s\cdot \sin(\alpha))}(x,y,z)\nonumber\\
	&=\left(e^{s\cdot \sin(\alpha)\sqrt{-\kappa}}x,y+s\cdot\cos(\alpha),z+s\cdot\sin(\alpha)\right)\label{translationek}
	\end{align}
	and the corresponding Killing field is
	\begin{equation}K_{(x,y,z)}=\frac{d}{ds}\Big\vert_{s=0}\Phi_s(x,y,z)=x\sin(\alpha)\sqrt{-\kappa}e^{-z\sqrt{-\kappa}}E_1+\cos(\alpha)E_2+\sin(\alpha)E_3.\label{killingek}\end{equation}
	In both cases we observe $K$ is independent of $y$.
	
	\subsubsection{Foliation by vertical planes transversal to \texorpdfstring{$c$}{c}}\label{foliation}
	Let $c$ be as in \eqref{ektgeodesic}, that is, $c$ is a geodesic in $\Ekt$ with slope $\alpha\in(0,\pi)$ whose projection $\tilde{c}:=\Pi\circ c$ in $(\R^2,\tilde{g})$ is also geodesic. We refer to \eqref{ektbase} for the definition of $\tilde{g}$. Then there is a foliation of $(\R^2,\tilde{g})$ by geodesics $(\tilde{\beta}_s)_{s\in\R}$ perpendicular to $\tilde{c}$ such that $\tilde{\beta}_s(0)=\tilde{c}(s)$ for all $s\in\R$. Therefore the vertical planes $P_s:=\Pi^{-1}\big(\tilde{\beta}_s\big)$ foliate $\Ekt$. As $\alpha\in(0,\pi)$, the geodesic $c$ is not vertical and thus meets each $P_s$ transversally. The left translations along $c$ satisfy $\Phi_s(0,0,0)=c(s)$, so that the foliation $(P_s)_{s\in\R}$ is invariant under $(\Phi_s)_{s\in\R}$. Moreover, we have $\Phi_s(P_0)=P_s$.
	
	In \autoref{chapter23} we will need these planes explicitly. For that purpose it is sufficient to exhibit the geodesic $\tilde{\beta}_0$ in $(\R^2,\tilde{g})$ perpendicular to 
	\[
	\tilde{c}(s)=(\Pi\circ c)(s)=\begin{cases}(0,s)&\mbox{if }\tau\neq0,\\
	(0,s\cdot \sin(\alpha))&\mbox{if }\tau=0.
	\end{cases}
	\]
	Compare with \eqref{ektgeodesic} and recall that $\Pi(x,y,z)=(x,z)$.
	\begin{Prop}Consider
		\begin{equation}
		\tilde{\beta}_0\colon\R\to\R^2,\qquad \tilde{\beta}_0(t):=\begin{cases}\left(\frac{\tanh(t\sqrt{-\kappa})}{\sqrt{-\kappa}}, \frac{\log(\sech(t\sqrt{-\kappa}))}{\sqrt{-\kappa}}\right)&\text{for }\kappa<0,\\
		(t,0)&\text{for }\kappa=0.
		\end{cases}\label{geodpar}
		\end{equation}
		Then $\tilde{\beta}_0(t)$ is a continuous function of $\kappa$: For each $t\in\R$ the limit of $\tilde{\beta}_0(t)$ for $\kappa<0$ and $\kappa\to0$ exists and equals $(t,0)$. Moreover, $\tilde{\beta}_0$ is a unit-speed geodesic in $(\R^2,\tilde{g})$, where $\tilde{g}$, as defined in \eqref{ektbase}, is the metric induced by $\Pi\colon \R^2\ltimes_{A(\kappa,\tau)}\R\to\R^2$. Each horizontal lift $\beta$ of $\tilde{\beta}_0$ satisfies
		\begin{equation}
		\beta^\prime(t)=\sech\big(t\sqrt{-\kappa}\big)E_1-\tanh\big(t\sqrt{-\kappa}\big)E_3.\label{geodder}
		\end{equation}
	\end{Prop}
	

	\begin{proof}[Sketch of proof]The claim about the continuity of $\tilde{\beta}_0(t)$ is clear. Let us recall the metric $\tilde{g}$ on $\R^2$:
		\begin{equation*}
		\tilde{g}_{(x,z)}=e^{-2z\sqrt{-\kappa}}dx^2+dz^2.
		\end{equation*}
		
		For $\kappa=0$ we have $\tilde{\beta}_0(t)=(t,0)$ and the metric induced on $\R^2$ is the Euclidean one, so that $\tilde{\beta}_0$ is geodesic.
		
		For $\kappa<0$ we consider the upper half-plane $U:=\{(u,v)\colon v>0\}$ and note that \[g_{(u,v)}:=\frac{1}{-\kappa v^2}\langle\cdot,\cdot\rangle_{\R^2}\]
		defines a metric of constant sectional curvature $\kappa$ on $U$. Then
		\[
		\R\to U,\qquad t\mapsto \big(\tanh\big(t\sqrt{-\kappa}\big),\sech\big(t\sqrt{-\kappa}\big)\big)
		\]
		parametrises a unit-speed geodesic semi-circle through $(0,1)$.
		One can check that
		\[
		\varphi\colon\left(\R^2,\tilde{g}\right)\to\left(U,g\right),\qquad \varphi(x,z):=\left(x\sqrt{-\kappa},e^{z\sqrt{-\kappa}}\right)
		\]
		is an isometry with 
		\[
		\varphi^{-1}\colon\left(U,g\right)\to\left(\R^2,\tilde{g}\right),\qquad\varphi^{-1}(u,v)=\left(\frac{u}{\sqrt{-\kappa}},\frac{\log(v)}{\sqrt{-\kappa}}\right).
		\]
		Applying $\varphi^{-1}$ to the geodesic in $U$ proves the claim about $\tilde{\beta}_0$.
		Regarding the horizontal lift $\beta$ we observe the following for $v:=\sech\big(t\sqrt{-\kappa}\big)E_1-\tanh\big(t\sqrt{-\kappa}\big)E_3$:
		\begin{itemize}
			\item $v$ is horizontal,
			\item $\nabla_vv\equiv0$ and
			\item $d\Pi\, v\equiv \tilde{\beta}_0^\prime$.
		\end{itemize}
		This completes the proof.\end{proof}
	
	\subsubsection{Commutator relations of left translations along \texorpdfstring{$c$}{c}}
	The family $(\Phi_s)_{s\in\R}$ of left translations along a geodesic $c$ whose projection is also a geodesic commutes in the following way with other isometries:
	
	\begin{Prop}\label{propcommut}Let $c$ be as in \eqref{ektgeodesic}, that is, $c$ is a geodesic with slope $\alpha\in(0,\pi)$ in $\Ekt$ and its projection $\tilde{c}=\Pi\circ c$ in $(\R^2,\tilde{g})$ is geodesic. Let $\Phi_s=\Lcal_{c(s)}$ be the left translation along $c$.
		\begin{enumerate}[(a)]
			\item For $\tau=0$ reflection $\sigma$ through the vertical plane $\Pi^{-1}\left(\tilde{c}\right)$ commutes with $(\Phi_s)_{s\in\R}$, i.e., $\Phi_s\circ\sigma=\sigma\circ\Phi_s$.
			
			\item For $\tau\neq0$ the geodesic $c$ is horizontal and the half-turn rotation $\rho_c$ about $c$ commutes with $(\Phi_s)_{s\in\R}$, that is, $\rho_c\circ\Phi_s=\Phi_s\circ\rho_c$.
			
			\item For any $\tau\in\R$ we have the following:
			\begin{itemize}
				\item Vertical translations $(T_t)_{t\in\R}$, see \eqref{ektvtranslation}, and $(\Phi_s)_{s\in\R}$ commute: for $s,t\in\R$ we have $T_t\circ\Phi_s=\Phi_s\circ T_t$.
				
				\item The horizontal lift $\beta$ of $\tilde{\beta}_0$ with $\beta(0)=c(0)$ is a horizontal geodesic and the rotation of angle $\pi$ about $\beta$, denoted by $\rho_0$, satisfies $\rho_0\circ\Phi_s=\Phi_{-s}\circ\rho_0$.
			\end{itemize}
		\end{enumerate}
	\end{Prop}
	\begin{proof}
		
		(a): This relation is easily verified by noting that $\sigma(x,y,z)=(-x,y,z)$ is the reflection through $\Pi^{-1}\left(\tilde{c}\right)$. Looking at \eqref{translationek}, we see $\Phi_s\circ\sigma=\sigma\circ\Phi_s$.
		
		(b): Since $\Ekt=\R^2\ltimes_{A(\kappa,\tau)}\R$ is a metric semi-direct product, the half-turn rotation $\rho_c$ about $c(s)=(0,0,s)$ is given by $\rho_c(x,y,z)=(-x,-y,z)$. We refer to \cite[Section 2.3.]{MP} for this explicit expression. Using this and \eqref{translationekt} the commutator relation follows readily.
		
		(c): The claim about vertical translations commuting with left translations along $c$ follows directly from the explicit formulae \eqref{ektvtranslation}, \eqref{translationekt} and \eqref{translationek}.
		
		For the last claim we argue differently, because an explicit computation of $\rho_0$ is possible but tedious for $\kappa<0$. We make use of the following fact: An isometry of a connected Riemannian manifold is uniquely determined by its differential at one point.
		
		The curves $\beta$ and $c$ intersect at $(0,0,0)$, so that the half-turn rotation $\rho_0$ about $\beta$ satisfies $\rho_0(c(s))=c(-s)$. Thus we have \[(\rho_0\circ\Phi_s\circ\rho_0^{-1})(0,0,0)=\rho_0(\Phi_s(0,0,0))=\rho_0(c(s))=c(-s)=\Phi_{-s}(0,0,0)\]
		It remains to show that the differentials at $(0,0,0)$ are equal. Since $\rho_0$ fixes $\beta$, we introduce the following left invariant frame:
		\begin{align*}
		F_1&:=\sech(t\sqrt{-\kappa})E_1-\tanh(t\sqrt{-\kappa})E_3\overset{\eqref{geodder}}{=}\beta^\prime(t),\\
		F_2&:=\tanh(t\sqrt{-\kappa})E_1+\sech(t\sqrt{-\kappa})E_3,\\
		F_3&:=E_2.	
		\end{align*}
		Each $F_j$ is left invariant, so we have $d\Phi_s\vert_{(0,0,0)}(F_j)=F_j$ and $d\Phi_{-s}\vert_{(0,0,0)}(F_j)=F_j$ for $j\in\{1,2,3\}$. The half-turn rotation $\rho_0=\rho_0^{-1}$ about $\beta$ satisfies \[d\rho_0\vert_p(F_1)=F_1,\quad d\rho_0\vert_p(F_2)=-F_2\quad\mbox{and}\quad d\rho_0\vert_p(F_3)=-F_3.\] 
		This implies $d(\rho_0\circ\Phi_s\circ\rho_0)\vert_{(0,0,0)}=d\Phi_{-s}\vert_{(0,0,0)}$ and the claim follows.	
	\end{proof}

	\subsection{Constant mean curvature spheres}\label{ektspheres}
	
	The study of \mcH-surfaces, and in particular of \mcH-spheres, in $E(\kappa,\tau)$-spaces or a more general homogeneous three-manifold $X$ depends on the \emph{critical mean curvature}. We define it first as in \cite[Definition 1.1]{MeeksCheeger}:
	
	\begin{Def}Let $X$ be a simply connected non-compact homogeneous three-manifold and let $\mathcal{A}$ be the collection of all compact, immersed orientable surfaces in $X$. For a given surface $\Sigma\in\mathcal{A}$ let $\lvert H_\Sigma\rvert\colon\Sigma\to [0,\infty)$ stand for the absolute mean curvature function of $\Sigma$. The \emph{critical mean curvature} of $X$ is defined as
		\begin{equation}
		H(X):=\inf\left\{\max_{\Sigma}\lvert H_\Sigma\rvert\colon\Sigma\in\mathcal{A}\right\}.\label{critical}
		\end{equation}

	\end{Def}
	For the homogeneous three-manifolds occurring in this paper, the critical mean curvature is
	\[
	H(X)=\begin{cases}
	\frac{\sqrt{-\kappa}}{2}&\mbox{if }X=\Ekt\mbox{ with }\kappa\leq0,\\
	0&\mbox{if }X=\Sol.
	\end{cases}
	\]
	This follows from \cite[Theorem 3.32]{MP}.
	The behaviour of \mcH-surfaces with $H=0$, $0<H<H(X)$, $H=H(X)$ and $H>H(X)$ is very different. We refer to \cite{abresch2004}, \cite{Abresch}, \cite{Daniel} and \cite{Notes2009} for various results in that regard.
	
	\mcH-spheres $S_H$ exist for all $H>H(E)$ in a non-compact $E(\kappa,\tau)$-space $E$ and they have the following properties: 
	\begin{enumerate}[1.]
		\item A sphere $S_H$ is embedded and unique up to isometries.
		\item Any sphere has a centre $p$, that is, we have $S_H=S_H(p)$.
		\item  Each ambient isometry fixing the centre leaves $S_H(p)$ invariant. In particular, $S_H(p)$ is invariant by rotations about the fibre through $p$.
	\end{enumerate}
	We refer to \cite{MeeksTransversality} and \cite{Meeks2017} for these properties. In the following, we are not going to need any additional properties as it was the case in $\Sol$ (see \autoref{solsphere}). The reason is that for $\tau\neq0$ we do not have a two-dimensional subgroup orthogonal to the fibres in $\Ekt$. 
	
	We note that some properties of constant mean curvature spheres in $E(\kappa,\tau)$-spaces follow from \cite[Theorem 6]{Abresch}. This shows that any immersed constant mean curvature sphere in a non-compact $E(\kappa,\tau)$-space is a rotational sphere, which also implies uniqueness up to isometries. Explicit examples of rotationally invariant \mcH-spheres in $\mathbb{H}^2\times\R$, $\mathbb{S}^2\times\R$, $\Nil_3$ and $\widetilde{\PSL}_2(\R)$ can be found in the following papers: \cite{Hsiang, Onnis, SaEarp,  FMP, Penafiel}.
	
	The assumption on $E$ being non-compact is crucial for embeddedness. Examples of non-embedded rotationally invariant \mcH-spheres in the Berger spheres have been found by \cite{Torralbo}.
	

	
	\section{Invariant \texorpdfstring{\mcH}{mcH}-cylinders whose axis is a geodesic with geodesic projection in \texorpdfstring{$E(\kappa,\tau)$}{E(k,t)}-spaces with \texorpdfstring{$\kappa\leq0$}{k<=0}}\label{chapter22}
	In this section we work in $E(\kappa,\tau)$-spaces with $\kappa\leq0$. As a model, we use $\Ekt$ as introduced in \autoref{sec42}. For a geodesic $c$ whose projection is also a geodesic, see \eqref{ektgeodesic}, we consider the left translation $\Phi_s=\Lcal_{c(s)}$ along $c$ and study $(\Phi_s)_{s\in\R}$-invariant \mcH-surfaces in $\Ekt$. We carry over the arguments used in \autoref{chapter12} of the \hyperref[chapter1]{first part}:
	\begin{itemize}
		\item As in case of $\Sol$, we consider $(\Phi_s)_{s\in\R}$-invariant surfaces whose generating curves are graphical. They will be vertical graphs.
		
		\item The geometric discussion of the ODE for the graphical solution and its extension to a simple closed embedded curve carry over from $\Sol$ almost literally except for one argument, so that we only state what is different.
	\end{itemize}

	\subsection{ODE for \texorpdfstring{\mcH}{mcH}-surfaces invariant under left translation along \texorpdfstring{$c$}{c}}
	We recall some notation from \autoref{sec42}. Let $c$ be as in \eqref{ektgeodesic}, that is, it is a geodesic with slope $\alpha\in(0,\pi)$ in $\Ekt$ whose projection $\tilde{c}=\Pi\circ c$ in $\R^2$ is also a geodesic. Through each point $\tilde{c}(s)$ there is a geodesic $\tilde{\beta}_s$ perpendicular to $\tilde{c}$. These geodesics foliate $\R^2$ and thus
	the vertical planes $(P_s)_{s\in\R}$ defined by $P_s=\Pi^{-1}\big(\tilde{\beta}_s\big)$ foliate $\Ekt$. Moreover, this foliation is transversal to $c$ and it is preserved by left translation $\Phi_s$ along $c$. 
	
	Finally, let $\beta$ be the horizontal lift of $\tilde{\beta}_0$ such that $\beta(0)=c(0)$. We also recall the vertical translation $(T_s)_{s\in\R}$, see \eqref{ektvtranslation}. For $\mathcal{C}^2$-functions $x,y\colon J\to\R$ consider the unit-speed curve 
	\[\gamma\colon J\to \Ekt\mbox{,}\qquad \gamma(t):=T_{y(t)}\big(\beta(x(t))\big),\] which is contained in the vertical plane $P_0$. A surface invariant by left translation along $c$ is parametrised by
	\begin{equation}f\colon \R\times J\to \Ekt,\qquad f(s,t):=\Phi_s(\gamma(t))=\Phi_s(T_{y(t)}(\beta(x(t)))).\label{surfaceekt}\end{equation}
	We specialise to $x(t)=t$ and $h(t)=y(t)$, that is, we are considering vertical graphs over $\beta$. For these vertical graphs over $\beta$ we study the ODE for constant mean curvature:
	\begin{Prop}\label{odeekt}Let $H\in\R$.
		There exists a smooth function $F\colon\R^2\to\R$ such that the invariant surface
		\[
		f\colon\R\times J\to \Ekt,\qquad f(s,t):=\Phi_s\big(T_{h(t)}(\beta(t))\big),\qquad\text{where }h\in\mathcal{C}^2(J,\R),
		\]
		has constant mean curvature $H$ with respect to the inner normal if and only if 
		\begin{equation}h^{\prime\prime}(t)=F(t,h^\prime(t))\qquad \text{for all t }\in J.\label{odepropekt}\end{equation}
	\end{Prop}
	
	\begin{proof}Let $v_1:=\partial_sf$ and $v_2:=\partial_tf$. We denote the inner normal to $f$ by $N$, so that $g_{ij}:=\langle v_i,v_j\rangle$ and $b_{ij}:=\langle \nabla_{v_i}v_j,N\rangle$ for $i,j\in\{1,2\}$ are the coefficients of the first and second fundamental form. Then the mean curvature of $f$ is given by
		\[
		H=\frac{b_{11}g^{11}+2b_{12}g^{12}+b_{22}g^{22}}{2}.
		\]
		Here we note that $H$ depends on $t$, $h^\prime(t)$ and $h^{\prime\prime}(t)$, but not on $h(t)$ itself. This is due to the existence of vertical translations $(T_s)_{s\in\R}$ commuting with $(\Phi_s)_{s\in\R}$; see \autoref{propcommut} (c).
		
		We assume $H$ to be constant and therefore get an implicit differential equation depending on $h^\prime(t)$ and $h^{\prime\prime}(t)$. Now we want to show we can solve this implicit equation for $h^{\prime\prime}(t)$. We have
		\[
		v_2=\beta^\prime+h^\prime\xi\quad\text{and}\quad\nabla_{v_2}v_2=\underbrace{\nabla_{\xi}\beta^\prime+h^\prime\nabla_{\beta^\prime}\xi}_{=:w}+h^{\prime\prime}\xi.
		\]
		We obviously have $w=w(t,h^\prime(t))$ and so the only term containing $h^{\prime\prime}(t)$ is 
		\[
		\frac{b_{22}g_{11}}{2\det(g)}= \frac{\langle \nabla_{v_2}v_2,N\rangle g_{11}}{2\det(g)}=\frac{\left(h^{\prime\prime}\langle N,\xi\rangle+\langle w,N\rangle\right)g_{11}}{2\det(g)}.
		\]
		The surface $f$ is a Killing graph with respect to the Killing field $\xi$, so that $\langle N,\xi\rangle$ is positive for $N$ is the inner normal. We also have $g_{11}>0$ since $(\Phi_s)_{s\in\R}$ does never act trivially. Hence we can solve the implicit equation for $h^{\prime\prime}$ and obtain a function $F\colon\R^2\to\R$ with ${h^{\prime\prime}(t)=F(t,h^\prime(t))}$. This function $F$ is smooth because each $\Phi_s$ is smooth and thus are $g$ and $b$. It is defined on whole $\R^2$ because we can prescribe any kind of interval $J$ and function $h\colon J\to\R$.
	\end{proof}
	
	\subsection{Geometric discussion of the ODE: Half-cylinder solution and its extension to an embedded \texorpdfstring{\mcH}{mcH}-cylinder with axis \texorpdfstring{$c$}{c}}
	The following proposition corresponds to \autoref{lemode}:
	
	\begin{Prop}\label{lemodeekt}Given $a\in\R$ and $H>H(\Ekt)$, the Picard-Lindel\"of Theorem gives a unique maximal solution $h\colon I_{\max}\to\R$ with $h(0)=a$ and $h^\prime(0)=0$ satisfying the ODE \eqref{odepropekt}. For each $a\in\R$ it has the following properties:
		\begin{enumerate}[(a)]
			\item \emph{[Symmetry]:} There is $R=R(a)\in(0,\infty]$ such that we have $I_{\max}=(-R,R)$ and $h(t)=h(-t)$ for all $t\in(-R,R)$.
			
			\item \emph{[Bounded existence interval]:} We have $R(a)<\infty$.
			
			\item \emph{[Bounded vertical graph]:} There is $K=K(a)>0$ such that $\lvert h(t)\rvert\leq K$ for all $t\in[-R,R]$ where $h(\pm R):=\lim_{t\to\pm R}h(t)$.
		\end{enumerate}
		
		
		
		
	\end{Prop}
	
	\begin{proof}We consider $c$ as in \eqref{ektgeodesic}. The surface $f(s,t)=\Phi_s(T_{h(t)}(\beta(t)))$ invariant by left translation $\Phi_s=\Lcal_{c(s)}$ along $c$ is generated by the curve $\gamma(t):=T_{h(t)}(\beta(t))$, which is a vertical graph over $\beta$. The curve $\beta$ satisfies $\beta(0)=c(0)$ and it is the horizontal lift of $\tilde{\beta}_0$ intersecting $\tilde{c}=\Pi\circ c$ orthogonally. Thus we have $\gamma^\prime=\beta^\prime+h^\prime\xi$. 
		
		(a): For $h^\prime(0)=0$ the tangent $\gamma^\prime(0)$ is horizontal since $\beta$ is horizontal. As the horizontal lift of a geodesic, $\beta$ is a geodesic itself. 
		
		If $\tau=0$, then reflection through the vertical plane $\Pi^{-1}\left(\tilde{c}\right)$ is an isometry. This reflection and $(\Phi_s)_{s\in\R}$ commute, see \autoref{propcommut} (a), and so the reflected graph satisfies the same ODE. Moreover the initial values are invariant. Hence the reflection leaves the solution invariant. 
		
		In case $\tau\neq0$ the curve $c$ is a horizontal geodesic. Let us translate the graph such that $a=0$. Applying the half-turn rotations $\rho_c$ about $c$ and $\rho_0$ about $\beta$ yields a graphical solution satisfying the same ODE (by \autoref{propcommut} (b) and (c)) and the initial values remain invariant under these rotations.
		
		In both cases we conclude that the existence interval and the solution $h$ are symmetric.
		
		(b) and (c): These properties follow as those in \autoref{lemode}. For (b) we move spheres by left translations along a fibre instead of a $y$-axis in $\Sol$. For (c) we move spheres by left translations along $T_v\circ\beta$ for some $v\in\R$ instead of an $x$-axis in $\Sol$.\end{proof}
	
	Recall that \eqref{odepropekt} does not depend on $h(t)$. In case of $\Sol$ we needed a height zero solution because the ODE \eqref{odesol2} depends also on $h(t)$. In the present situation, we simply apply a vertical translation to a symmetric solution from \autoref{lemodeekt} and obtain $h(\pm R)=0$. We still need to ensure $h(t)<0$ for all $t\in(-R,R)$.
	
	\begin{Lem}[Height zero half-cylinder solution]\label{ektzeroheight}
		There exists a unique $a_0<0$ such that the maximal solution of \eqref{odepropekt} with $h(0)=a_0$ and $h^\prime(0)=0$ has the following properties: We have $h(\pm R(a_0))=0$ and $h(t)<0$ for all $t\in(-R(a_0),R(a_0))$. Moreover, the solution satisfies $\lim_{t\to\pm R}h^\prime(t)=\pm\infty$.
	\end{Lem}
	\begin{proof}In the proof of the corresponding \autoref{zeroheight} in $\Sol$, we made use of specific properties of \mcH-spheres in $\Sol$. Namely that the minimal value on the $y$-coordinate is attained on the $y$-axis through the centre of a sphere. Since this property is not available in a general $\Ekt$, we need different surfaces for the comparison argument. 
		
		For $r\in\R$ we define a surface by
		\begin{equation}
		\varphi_r\colon\R\times\R\to\Ekt,\qquad \varphi_r(s,t):=\Phi_s(T_r(\beta(t))).\label{ektminimal}
		\end{equation}
		This surface is invariant by left translation $\Phi_s$ along $c$. We denote the image of $\varphi_r$ by $M_r$. Since vertical translation $T_r$ and $\Phi_s$ commute (by \autoref{propcommut} (c)), we see that $M_r=T_r(M_0)$ and thus the family $(M_r)_{r\in\R}$ is a foliation of $\Ekt$.
		
		Each surface $M_r$ is foliated by horizontal geodesics because $\beta$ is a horizontal geodesic.  Through each point of $M_r$ there is a horizontal geodesic $\beta_{s_0}\subseteq M_r$ passing through $c(s_0)$. In \autoref{propcommut} we denoted the half-turn rotation about $\beta$ by $\rho_0$. Let $\rho_{s_0}$ be the half-turn rotation about $\beta_{s_0}$. Due to $\beta_{s_0}=\Phi_{s_0}\circ\beta$ we have $\rho_{s_0}=\Phi_{s_0}\circ\rho_0\circ\Phi_{-s_0}$, so that $\rho_{s_0}$ and $(\Phi_s)_{s\in\R}$ satisfy $\rho_{s_0}\circ\Phi_s=\Phi_{-s}\circ\rho_{s_0}$ (compare with \autoref{propcommut} (c)). This implies the $(\Phi_s)_{s\in\R}$-invariant surface $M_r$ is invariant by the half-turn rotation $\rho_{s_0}$. Since this is true for any point on $M_r$, we see that $M_r$ is a minimal surface. It is embedded because it is a $(\Phi_s)_{s\in\R}$-invariant surface.
		
		For the actual proof of this lemma, let us denote the image of the $(\Phi_s)_{s\in\R}$-invariant \mcH-surface $f(s,t)=\Phi_s(T_{h(t)}(\beta(t)))$ by $\Sigma$. For more details regarding the notation, see the proof of \autoref{lemodeekt}. 
		
		\emph{First claim:} For $a\in\R$ we have $h(\pm R(a))\geq a$. Suppose this were false, that is, $h(\pm R(a))<a$. Then the boundary of $\Sigma$ is contained in $M_u$ for some $u< a$. There exists $v\in\R$ such that $M_v$ is contained in the mean convex side of $\Sigma$. Since the surfaces $(M_r)_{r\in\R}$ foliate $\Ekt$, we can move $M_v$ towards $\Sigma$. The boundary of $\Sigma$ is contained in $M_u$ for $u<a=h(0)$ so that we get a tangential intersection in the interior. See also \autoref{ektlemmafirstclaim} below. By the maximum principle this could only be possible for $H_\Sigma\leq0=H_{M_r}$, in contradiction to $H_\Sigma>H(\Ekt)>0$.
		
		\begin{figure}
			\begin{center}
				\includegraphics{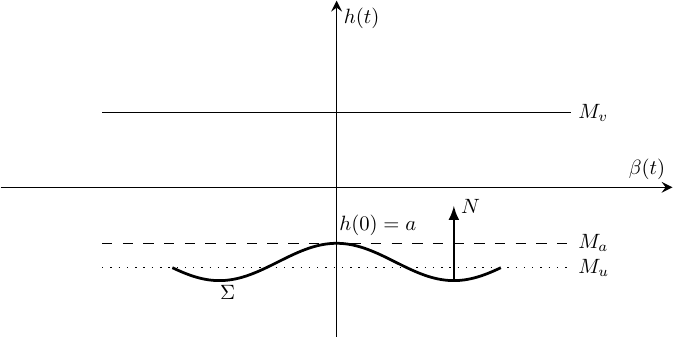}
			\end{center}
			\caption{\autoref{ektzeroheight}: comparison argument used in the ''First claim``}
			\label{ektlemmafirstclaim}
		\end{figure}

		\emph{Second claim:} For $a\in\R$ we have $h(\pm R(a))>a$. We argue by contradiction and in view of the previous claim we may assume $h(\pm R(a))=a$. Lets distinguish three cases:
		
		\begin{enumerate}[I.]
			\item There exists $t_0\in(-R(a),R(a))$ with $h(t_0)>a$. In this case we argue as in the first claim.
			
			\item There exists $t_0\in(-R(a),R(a))$ with $h(t_0)<a$. By symmetry of $h$ we can assume $t_0>0$. The restricted surface $\widetilde{\Sigma}:=f(\R,(-t_0,t_0))$ then has its boundary contained in $M_u$ for $u<a=h(0)$. With respect to $\widetilde{\Sigma}$, we are again in the situation from the first claim and obtain a contradiction.
			
			\item We have $h(t)=a$ for all $t\in(-R(a),R(a))$. Then $\Sigma$ is a piece of the minimal surface $M_a$, which contradicts $H>H(\Ekt)>0$.
		\end{enumerate}
		
		From these two claims we find, by applying a vertical translation, a unique $a_0<0$ such that $h(\pm R(a_0))=0$.
		
		\emph{Third claim:} For this $a_0$ we have $h(t)<0$ for all $t\in(-R(a_0),R(a_0))$. On the contrary, if it were false, there would exist $t_0>0$ (due to symmetry of $h$ and the previous claims) with $h(t_0)\geq 0$. The case $h(t_0)>0$ is ruled out as in the first claim. For $h(t_0)=0$ we distinguish the following three cases:
		\begin{enumerate}[I.]
			\item There exists $t_1\in(t_0,R(a_0))$ with $h(t_1)>0$, which is ruled out as before.
			
			\item If there exists $t_1\in(t_0,R(a_0))$ with $h(t_1)<0$, we consider the restricted surface $\widehat{\Sigma}:=f(\R,(-t_1,t_1))$. Its boundary is contained in $M_u$ for some $u<0=h(t_0)$. The comparison argument from the first claim is then applicable with respect to the surface $\widehat{\Sigma}$.
			
			\item The case $h(t)=0$ for all $t\in(t_0,R(a_0))$ cannot occur because this would be a piece of the minimal surface $M_0$.
		\end{enumerate}
		
		The claimed asymptotic behaviour now follows as in \autoref{zeroheight} (c).
	\end{proof}
	The solution from \autoref{ektzeroheight} generates a \emph{horizontal \mcH-cylinder} for all $\tau$ and an \mcH-cylinder with sloped axis for $\tau=0$, that is, the following main result includes tilted \mcH-cylinders in $\mathbb{H}^2\times\R$ and horizontal \mcH-cylinders in $\widetilde{\PSL}_2(\R)$:
	\begin{Theorem}\label{thmcylekt}We consider $\Ekt$ for $\kappa\leq0$ with fibration $\Pi\colon\Ekt\to(\R^2,\tilde{g})$. Let $\tilde{c}$ and $\tilde{\beta}$ be orthogonal geodesics in $(\R^2,\tilde{g})$ such that the geodesic $c$ in $\Ekt$ as in \eqref{ektgeodesic} is a geodesic whose $\Pi$-projection is $\tilde{c}$. Let $(\Phi_s)_{s\in\R }$ be the family of left translations along $c$. 
		
		For each $H>H(\Ekt)=\frac{\sqrt{-\kappa}}{2}$ there is a smooth embedded simple closed curve $\gamma$ in the vertical plane $\Pi^{-1}\big(\tilde{\beta}\big)$ which generates an embedded $(\Phi_s)_{s\in\R}$-invariant cylinder $f(s,t)=\Phi_s(\gamma(t))$ with constant mean curvature $H$.	
		
		For arbitrary $\tau$ the surface is invariant by a half-turn rotation about $\beta$, the horizontal lift of $\tilde{\beta}$ with $\beta(0)=c(0)$. If the axis $c$ is horizontal, the surface is invariant by a half-turn rotation about its axis $c$. For $\tau=0$ the surface has a vertical mirror plane containing the axis $c$.  
	\end{Theorem}\index{cylinder in $E(\kappa,\tau)$}
	
	\begin{proof}Let $h\colon (-R,R)\to\R$ be the solution from \autoref{ektzeroheight}. Thus the graph of $h$ meets $\beta$ orthogonally at $t=\pm R$. We extend the graph of $h$ by $\rho_0$ to a closed curve $\gamma$, where $\rho_0$ denotes the half-turn rotation about $\beta$. The curve $\gamma$ is smooth because of the graph's asymptotic behaviour and $\gamma$ is embedded because of $h(t)<0$ for all $t\in(-R,R)$. Due to $\rho_0\circ\Phi_s=\Phi_{-s}\circ\rho_0$ from \autoref{propcommut}, the curve $\gamma$ is generating a $(\Phi_s)$-invariant surface with constant mean curvature $H$. The surface is embedded because $\gamma$ is and due to the form of the left translations $(\Phi_s)_{s\in\R}$. The claimed symmetries follow from \autoref{propcommut} (a) and (b).\end{proof}
	
	\begin{Rem}\begin{enumerate}[1.]
			\item In his Ph.D. thesis, \cite{Penafiel} studied various invariant surfaces in the space $E(-1,\tau)=\widetilde{\PSL}_2(\R)$. A one-parameter family he considers is translation along a horizontal geodesic in $E(-1,\tau)$, corresponding to the left translations we considered in $\mathbb{E}(-1,\tau)$. He chose the upper-half plane model and the vertical plane containing the generating curve $\gamma$ is \[
			P_0=\{(\cos(\theta),\sin(\theta),h)\colon \theta\in(0,\pi)\mbox{ and }h\in\R\}.
			\]
			He considers graphs $h=h(\theta)$ generating an invariant \mcH-surface. A flux computation, see \cite[Lemma 8.1.2]{Penafiel}, yields the representation
			\[
			h(\theta)=\int \frac{(d-2H\cot(\theta))\sqrt{1+4\tau^2\cos^2(\theta)}}{\sqrt{1-\sin^2(\theta)(d-2H\cot(\theta))^2}}\,d\theta-2\tau\theta,
			\]
			where $d$ is a real number.
			
			For some values of $H$, $\tau$ and $d$, the integral can be computed explicitly. However, for $H>\frac{1}{2}$ and $\tau\neq0$ it seems that it has not been the case. With the help of Mathematica it is possible to represent $h$ in terms of elliptic integrals, though. 
			\item In $\mathbb{S}^2\times\R$ there exist also \mcH-cylinders with arbitrary geodesic axis. These occur as screw-motion surfaces and have been constructed by \cite{SaEarp}. 
			
			\item A complete classification of invariant \mcH-surfaces in $\mathbb{H}^2\times\R$ is known due to \cite{Onnis}. It includes the examples in $E(-1,0)$ from \autoref{thmcylekt}.	
		\end{enumerate}
	\end{Rem}
	
	\section{Horizontal diameter of an \texorpdfstring{\mcH}{mcH}-cylinder with horizontal axis}\label{chapter23}
	We constructed the \mcH-cylinders in $\Ekt$ for $\kappa\leq0$ from \autoref{thmcylekt} as follows: We fixed a geodesic $c$ in $\Pi\colon\Ekt\to\R^2$ whose projection $\tilde{c}=\Pi\circ c$ is a geodesic. The curve $\beta$ is the horizontal geodesic in $\Ekt$ with $\beta(0)=c(0)$ and whose $\Pi$-projection $\tilde{\beta}_0$ is orthogonal to $\tilde{c}$. For \mcH-surfaces invariant by left translation $\Phi_s$ along $c$ we obtained the height zero half-cylinder solution $h\colon(-R,R)\to\R$ from \autoref{ektzeroheight}. The function $h$ amounts to the curve
	\begin{equation}
	\Gamma\colon (-R,R)\to\Ekt,\qquad \Gamma(t):=T_{h(t)}(\beta(t)),\label{verticalgraph}
	\end{equation}
	where $T_s$ is the vertical translation \eqref{ektvtranslation} in $\Ekt$. This curve is a vertical graph with respect to the fibration $\Pi\colon\Ekt\to\R^2$. The $(\Phi_s)_{s\in\R}$-invariant surface generated by it is $f(s,t)=\Phi_s(\Gamma(t))$, compare with \eqref{surfaceekt}. We showed that $\Gamma$ can be extended to an embedded closed curve by a half-turn rotation about $\beta$.
	
	In this final section we compute the existence interval $I_{\max}=(-R,R)$ without calculating the actual ODE \eqref{odepropekt}. We apply a weight formula, that is, a flux computation, to achieve this. We only carry out the computation in case that $c$ is horizontal, but it works as well in the general case. Again, we are only working in $\Ekt$ for $\kappa\leq0$.
	
	It is convenient to parametrise the curve $\Gamma$ by arc-length:
	
	
	\begin{Lem}\label{lemweight}Let $H>H(\Ekt)$ and let $h\colon(-R,R)\to \R$ be the solution established in \autoref{ektzeroheight}. Then $\Gamma\colon(-R,R)\to\Ekt\mbox{, }\Gamma(t):=T_{h(t)}(\beta(t))$ is a curve, which is a vertical graph. A reparametrisation of $\Gamma$ by arc-length and with the same orientation gives a curve $\gamma\colon[0,L]\to \Ekt\mbox{, }\gamma=\beta\circ d+(0,e,0)$ for some $d,e\in\mathcal{C}^2([0,L],\R)$ with the following properties:
		\begin{itemize}
			\item $L$ is the arc-length of $\Gamma$ on $[0,R]$, that is $L=\int_0^R\sqrt{1+{h^\prime}^2(t)}\,dt$,
			\item $\gamma$ respects the initial values of the graph $\Gamma$, i.e., we have $\gamma(0)=(0,a_0,0)$ and $\gamma^\prime(0)=\beta^\prime(0)$, where $a_0$ is as in \autoref{ektzeroheight},
			\item $\gamma(L)=\beta(R)$ and $\gamma^\prime(L)=E_2$.
		\end{itemize}
		For the invariant surface $f\colon\R\times[0,L]\to E\mbox{, }f(s,t):=\Phi_s(\gamma(t))$, where $\Phi_s=\Lcal_{c(s)}$ and $c$ as in \eqref{ektgeodesic} is horizontal, i.e., has slope $\alpha=\pi/2$, the tangent vectors are
		\begin{align}
		v_1&:=\partial_sf=\sinh\Big(d(t)\sqrt{-\kappa}\Big)E_1+2\tau\sinh\Big(d(t)\sqrt{-\kappa}\Big)E_2+E_3\label{tangent1},\\
		v_2&:=\partial_tf=d^\prime(t)\sech\Big(d(t)\sqrt{-\kappa}\Big)E_1+e^\prime(t)E_2-d^\prime(t)\tanh\Big(d(t)\sqrt{-\kappa}\Big)E_3\label{tangent2}.
		\end{align}
	\end{Lem}
	
	\begin{proof}The claim about the reparametrisation is clear. For the tangent vector $v_1$ we have
		\[
		v_1=K_{\beta(d(t))+(0,e(t),0)}.
		\]
		Since the Killing field $K$ is independent of $y$ and $\beta$ is the horizontal lift of $\tilde{\beta}_0$, computed explicitly in \eqref{geodpar}, it suffices to insert \[x=\frac{\tanh(d(t)\sqrt{-\kappa})}{\sqrt{-\kappa}}\quad\mbox{ and }\quad z=\frac{\log(\sech(d(t)\sqrt{-\kappa}))}{\sqrt{-\kappa}}\] into $K_{(x,y,z)}$, given by \eqref{killingekt}, to show \eqref{tangent1}. For \eqref{tangent2} we note $v_2=d^\prime\beta^\prime\circ d+e^\prime E_2$ and refer to \eqref{geodder}.\end{proof}
	
	
	The horizontal diameter of a horizontal \mcH-cylinder in $\Ekt$ for $\kappa\leq0$ can be computed using the weight formula; it is independent of $\tau$.
	
	\begin{Theorem}\label{propweight}Let the geodesic $c$ as in \eqref{ektgeodesic} be horizontal. For the height zero solution $h\colon(-R,R)\to\R$ from \autoref{ektzeroheight} we then have
		\[R=\frac{1}{\sqrt{-\kappa}}\atanh\left(\frac{\sqrt{-\kappa}}{2H}\right),\]
		so that the \emph{horizontal diameter} of a horizontal \mcH-cylinder is $2R$. 
		
		The \mcH-cylinders with axis $c$, considered as a one-parameter family depending on $H\in(H(\Ekt),\infty)$, are unbounded for $H\to H(\Ekt)=\frac{\sqrt{-\kappa}}{2}$ and converge to the horizontal geodesic $c$ for the limit $H\to\infty$.\end{Theorem}
	
	\begin{proof}Let $h\colon (-R,R)\to\R$ be the height zero solution from \autoref{ektzeroheight} and let $\gamma$ be the reparametrisation of $\Gamma(t)=T_{h(t)}(\beta(t))$ by arc-length $L$ as in \autoref{lemweight}. We use the weight formula to determine the explicit value of $R$. By \autoref{lemweight} there are $d,e\in\mathcal{C}^2([0,L],\R)$ such that $\gamma=\beta\circ d+(0,e,0)$, that is, $\gamma(t)=T_{e(t)}(\beta(d(t)))$, where $T_s$ is the vertical translation \eqref{ektvtranslation}. Vertical translations and left translations $(\Phi_s)_{s\in\R}$ along $c$ commute by \autoref{propcommut} (c). Therefore the invariant surface $f$, which is given by $f\colon\R\times[0,L]\to\Ekt\mbox{, }f(s,t)=\Phi_s(\gamma(t))$, satisfies
		\[f(s,t)=\Phi_s(T_{e(t)}(\beta(d(t)))=T_{e(t)}(\Phi_s(\beta(d(t))))=\Phi_s(\beta(d(t)))+(0,e(t),0)).\]
		For a bounded domain $\Omega\subset\R\times[0,L]$ with $\partial\Omega$ a closed Jordan curve we let $\eta$ be the outer unit conormal along $f(\partial \Omega)$ and $N$ is the inner normal of the surface. The weight formula (see \cite[Proposition 3]{HLR} for a proof in a general Riemannian three-manifold) yields
		\begin{equation}2H\int_{f(\Omega)}\langle N,Y\rangle=\int_{f(\partial\Omega)}\langle\eta,Y\rangle,\qquad Y\textnormal{ Killing field}.\label{weight}\end{equation}
		We apply \eqref{weight} to the vertical Killing field $Y=\xi=E_2$ and set $\Omega:=[0,1]\times[0,L]$.\index{weight formula}\index{flux formula}
		
		We need some geometric data of the invariant surface $f$, which are easily computed with \autoref{lemweight}. For $v_1=\partial_sf$ and $v_2=\partial_tf$ and using \eqref{tangent1} respectively \eqref{tangent2}, we get:
		
		The entries of the induced metric $g=(\langle v_j,v_k\rangle)_{1\leq j,k\leq 2}$ on $\R\times [0,L]$ are
		\begin{equation}
		\begin{split}
		g_{11}&=\cosh^2(d\sqrt{-\kappa})+4\tau^2\sinh^2(d\sqrt{-\kappa}),\\
		g_{12}&=2\tau \sinh(d\sqrt{-\kappa})e^\prime,\\
		g_{22}&={d^\prime}^2+{e^\prime}^2,\end{split}\label{ffekt}
		\end{equation}
		with
		\begin{equation}
		\det(g)=\cosh^2(d\sqrt{-\kappa})\left({d^\prime}^2+{e^\prime}^2+4\tau^2\tanh^2(d\sqrt{-\kappa}){d^\prime}^2\right).\label{gram}
		\end{equation}
		The inner normal $N$ to $f$ satisfies
		\begin{equation}
		\begin{split}
		\sqrt{\det(g)}N&=\cosh(d\sqrt{-\kappa})e^\prime\left(-\sech(d\sqrt{-\kappa})E_1+\tanh(d\sqrt{-\kappa})E_3\right)\\
		&+\cosh(d\sqrt{-\kappa})d^\prime E_2\\
		&-2\tau\sinh(d\sqrt{-\kappa})d^\prime \left(\tanh(d\sqrt{-\kappa})E_1+\sech(d\sqrt{-\kappa})E_3\right).
		\end{split}\label{normalekt}
		\end{equation}
		First we compute the left-hand side of \eqref{weight}. In view of \eqref{normalekt}, we get
		\[
		2H\int\limits_{f(\Omega)}\langle N,E_2\rangle=2H\int\limits_{[0,1]\times[0,L]}d^\prime(t)\cosh\big(d(t)\sqrt{-\kappa}\big)\,ds\,dt=\frac{2H}{\sqrt{-\kappa}}\sinh\big(R\sqrt{-\kappa}\big).
		\]
		\begin{figure}[ht]
			\begin{center}
				\includegraphics{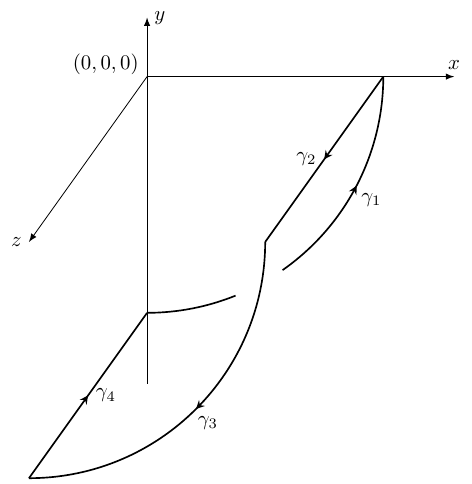}
			\end{center}
			\caption{Application of weight formula}
			\label{elemma}
		\end{figure}
		To compute the right-hand side of \eqref{weight}, we decompose the boundary parametrisation as \[f(\partial\Omega)=\gamma_1\oplus\gamma_2\oplus\gamma_3\oplus\gamma_4,\]
		where
		\[\gamma_1(t)=f(0,t),\quad \gamma_2(s)=f(s,L),\quad \gamma_3(t)=f(1,L-t),\quad \gamma_4(s)=f(1-s,0),\]
		see \autoref{elemma}. We denote by $\eta_1$ to $\eta_4$ the respective unit conormals along $\gamma_1$ to $\gamma_4$. Due to $\gamma_3(t)=\Phi_1(\gamma_1(L-t))$ we have $\gamma_3^\prime(t)=-\gamma_1^\prime(L-t)$ and thus ${\eta_3(t)=-\eta_1(L-t)}$. Since $E_2$ is a constant Killing field, this implies
		\[\int_{\gamma_1}\langle\eta_1,E_2\rangle+\int_{\gamma_3}\langle\eta_3,E_2\rangle=0.\]
		To determine the line integral $\int_{\gamma_4}\langle\eta_4,E_2\rangle$ note that $\gamma_4^\prime(s)=-\partial_sf(1-s,0)=E_3$ and $\partial_tf(1-s,0)=E_1$, i.e., we have \(\eta_4=E_1.\) This shows
		\[\int_{\gamma_4}\langle\eta_4,E_2\rangle=0.\]
		
		Finally we consider $\int_{\gamma_2}\langle\eta_2,E_3\rangle$. We note $\gamma_2^\prime(s)=v_1$ and for the conormal we get \[
		\eta_2=\frac{1}{\sqrt{g_{11}}\sqrt{\det(g)}}\left(-g_{12}v_1+g_{11}v_2\right).                                                                                                                             \]
		At $L$ we have 
		\[
		d(L)=R,\quad d^\prime(L)=0,\quad e(L)=0\quad\mbox{and}\quad e^\prime(L)=1,
		\]
		so that in view of \eqref{tangent1}, \eqref{tangent2} and \eqref{ffekt} evaluation at $L$ gives
		\begin{align*}\langle\eta_4,E_2\rangle\sqrt{g_{11}}&=\frac{-4\tau^2\sinh^2(R\sqrt{-\kappa})+\cosh^2(R\sqrt{-\kappa})+4\tau^2\sinh^2(R\sqrt{-\kappa})}{\cosh(R\sqrt{-\kappa})}\\
		&=\cosh(R\sqrt{-\kappa}).\end{align*}
		Noting that $\sqrt{\langle\gamma_4^\prime,\gamma_4^\prime\rangle}=\sqrt{\langle v_1,v_1\rangle}=\sqrt{g_{11}}$ we get
		\[\int_{\gamma_4}\langle\eta_4,E_2\rangle=\int_{[0,1]}\left[\langle\eta_4,E_2\rangle\cdot\sqrt{\langle\gamma_4^\prime,\gamma_4^\prime\rangle} \right]\,ds=\cosh(\sqrt{-\kappa}R).\]
		
		Combining these results yields
		\[\frac{2H}{\sqrt{-\kappa}}\sinh(R\sqrt{-\kappa})=\cosh(R\sqrt{-\kappa}).\]
		Because of $2H>\sqrt{-\kappa}$ we can solve this equation for $R$ and get
		\[
		R=\frac{1}{\sqrt{-\kappa}}\atanh\left(\frac{\sqrt{-\kappa}}{2H}\right).
		\]
		The unboundedness for $H\to H(E)$ is clear since $\atanh(u)$ is unbounded for $u\to1$. The convergence to $c$ for $H\to\infty$ follows by comparison with \mcH-spheres; here we use embeddedness of \mcH-cylinders in $E(\kappa,\tau)$-spaces with $\kappa\leq0$ and that \mcH-spheres in $\Ekt$ shrink to a point for $H\to\infty$ (see the references at the end of \autoref{ektspheres} for explicit examples).\end{proof}
	
	\begin{Rem}
		We have carried out the same computation for tilted \mcH-cylinders in the space $E(\kappa,0)$ with $\kappa\leq0$. The conormals along $\gamma_2$ and $\gamma_4$ turn out to be a bit more complicated but as a result we get 
		\[R=\frac{1}{\sqrt{-\kappa}}\atanh\left(\frac{\sqrt{-\kappa}}{2H}\right),\]
		as in \autoref{propweight}. We have not included the computation.
	\end{Rem}
	
	\section{Possible generalisations and open problems}
	Throughout the paper we have already indicated some open problems, see \autoref{emcylconj}, \autoref{nonembeddedcyl} and \autoref{emcylconj2}. In this final section we point out how techniques could be generalised to obtain new examples, and we mention related open problems.
	
	\subsection{Tilted \texorpdfstring{\mcH}{mcH}-cylinders in \texorpdfstring{$\widetilde{\PSL}_2(\R)$}{PSL(2,R)}}
	We use the model $\Ekt$ as introduced in \autoref{sec42} to describe $\widetilde{\PSL}_2(\R)$. For $\tau\neq0$ we set $\widetilde{\PSL}_2(R)=\mathbb{E}(-1,\tau)$. Then the base space $(\R^2,\tilde{g})$ with respect to \eqref{ektbase} is isometric to the hyperbolic plane with constant curvature $-1$. The critical mean curvature is $1/2$, so that we only consider \mcH-surfaces for $H>1/2$.
	
	By \autoref{slope}, a geodesic $c$ in $\widetilde{\PSL}_2(\R)=\mathbb{E}(-1,\tau)$ with slope $\alpha\in(0,\pi/2)$ projects onto a curve $\tilde{c}=\Pi\circ c$ of constant curvature $-\tau\cot(\alpha)$ in $(\R^2,\tilde{g})$. Say this curve is a geodesic circle in the hyperbolic plane $(\R^2,\tilde{g})$ with centre $\tilde{p}\in\R^2$. Emanating from $\tilde{p}$ are geodesic rays $\tilde{\beta}_s$ which intersect $\tilde{c}$ orthogonally in $\tilde{c}(s)$. Placing a curve $\gamma$ in the vertical plane $P_0=\Pi^{-1}\big(\tilde{b}_0\big)$, we can again define an invariant surface $f(s,t)=\Phi_s(\gamma(t))$, where $\Phi_s=\mathcal{L}_{c(s)}$ is the left translation along $c$. These left translations can be thought of as a kind of screw-motions in $\widetilde{\PSL}_2(\R)$. Is there a simple closed embedded curve $\gamma$ in $P_0$ which generates an \mcH-cylinder for $H>1/2$? Is it embedded? Note that in $\Nil_3=\mathbb{E}(0,\tau)$ such surfaces exist thanks to \cite{FMP}.
	
	Again, we choose a horizontal lift $\beta$ of $\tilde{\beta}_0$ such that $\beta(0)=c(0)$. For a vertical graph $\gamma(t)=T_{h(t)}(\beta(t))\subseteq P_0$ over $\beta$ we consider $f(s,t)=\Phi_s(\gamma(t))$ and require its mean curvature to be a constant $H>1/2$. This yields an ODE for $h$ of the form \[h^{\prime\prime}(t)=F(t,h(t),h^\prime(t)).\] The only isometry commuting with $\Phi_s$ is the half-turn rotation $\rho$ about $\beta$. Therefore we cannot expect a symmetric solution. Nevertheless, the comparison arguments with \mcH-spheres will still show boundedness of the existence interval $(R_-,R_+)$ and of the graph for any choice of initial values $h(0)=a$ and $h^\prime(0)=b$. The task is then to determine initial values $a_0$ and $b_0$ such that the graph of $h$ meets $\beta$ orthogonally and satisfies $h(R_\pm)=0$. If $h(t)<0$ for all $t\in(R_-,R_+)$ we can extend the vertical graph to an embedded simple closed curve generating a possibly self-intersecting tilted \mcH-cylinder in $\widetilde{\PSL}_2(\R)$. The argument from \autoref{chapter23} to compute the existence interval should carry over, but the computations will be more involved.
	
	\subsection{\texorpdfstring{\mcH}{mcH}-cylinders in general metric semi-direct products}
	
	In both parts the ambient spaces were metric semi-direct products $\R^2\ltimes_A\R$ where $A\in\operatorname{M}(2,\R)$. In these metric Lie groups \mcH-spheres $S_H$ exist for all $H>\frac{\operatorname{trace}(A)}{2}$, are unique up to isometries, have a centre $p$ and are Alexandrov-embedded by \cite{Meeks2017}. Note that Alexandrov-embeddedness still allows for comparison arguments involving the maximum principle.
	
	In $\R^2\ltimes_A\R$ any $z$-axis admits half-turn rotations; see \cite[Section 2.3.]{MP}. Lets fix $c(s)=(0,0,s)$ and consider $\Phi_s=\Lcal_{c(s)}$. One then has $\Phi_s(x,y,z)=(e^{As}(x,y),z+s)$, so that the foliation by $(x,y)$-planes $S_s=\{z=s\}$ is invariant under $(\Phi_s)_{s\in\R}$. Placing $\gamma(t)=(t,h(t),0)$ in $S_0$, we define a $(\Phi_s)_{s\in\R}$-invariant graphical surface by \[f(s,t):=\Phi_s(\gamma(t))\mbox{.}\] Unlike in $\Sol$, we have less symmetries at hand and a symmetric solution is not guaranteed. The situation is similar to that for tilted \mcH-cylinders in $\widetilde{\PSL}_2(\R)$: A bounded existence interval $(R_-,R_+)$ and a bounded graph will be guaranteed by comparison with \mcH-spheres. For a height zero half-cylinder solution different arguments are needed. If it exists, we can extend it by a half-turn rotation about $c$ to an embedded simple closed curve generating an \mcH-cylinder in $\R^2\ltimes_A\R$.
	
	\subsection{Unduloids in \texorpdfstring{$\Sol$}{Sol} and other homogeneous three-manifolds}
	Once existence translationally invariant \mcH-cylinders is settled, it is natural to look for more examples of properly embedded \mcH-annuli. Generalising unduloids to $\Sol$, $\mathbb{H}^2\times\R$, $\widetilde{\PSL}_2(\R)$ and $\Nil_3$ is therefore an interesting and difficult problem. In Euclidean space, an unduloid is a rotationally invariant \mcH-annulus about a geodesic axis $c$. It turns out to be a singly periodic surface with respect to a discrete subgroup of left translations along $c$. Likewise, an unduloid with geodesic axis $c$ in a metric Lie group $X$ is an Alexandrov-embedded \mcH-annulus with mean curvature $H>H(X)$ which is singly periodic with respect to a discrete subgroup of left translations along $c$.
	
	In $\mathbb{H}^2\times\R$, $\widetilde{\PSL}_2(\R)$ and $\Nil_3$ unduloids with a vertical geodesic axis exist as rotationally invariant \mcH-surfaces. Only in $\mathbb{H}^2\times\R$ existence of unduloids with a horizontal axis is known. They have been constructed by \cite{ManzanoTorralbo} using the Daniel correspondence from \cite{Daniel}, which relates \mcH-surfaces in $\mathbb{H}^2\times\R$ with mean curvature $H>1/2$ to minimal surfaces in a Berger sphere. So, for $\widetilde{\PSL}_2(\R)$ and $\Nil_3$ existence of horizontal unduloids is an open problem. In these spaces the difficulty is the lack of symmetries preserved by the Daniel correspondence.
	
	In \cite{Vrzina}, the author has studied the existence problem for tilted unduloids in $\mathbb{H}^2\times\R$ by means of the Daniel correspondence. Alexandrov reflection shows that any unduloid in $\mathbb{H}^2\times\R$ has a vertical mirror plane. The Daniel correspondence implies that, for the known examples, the corresponding minimal surfaces in a Berger sphere are embedded minimal annuli bounded by linked horizontal geodesics. If any pair of linked horizontal geodesics in a Berger sphere bounds exactly two embedded minimal annuli then existence of tilted unduloids in $\mathbb{H}^2\times\R$ follows.
	
	Finally, for $\Sol$ we expect unduloids to exist for the geodesic axes \eqref{solgeodesics} leading to horizontal and vertical \mcH-cylinders in $\Sol$. Unfortunately, a correspondence relating \mcH-surfaces with $H>0$ in $\Sol$ to minimal surfaces in another three-manifold is unknown, and a direct construction of unduloids in $\Sol$ seems very unlikely as it was in the case of \mcH-spheres. Since $\Sol$ admits two families of mirror planes, the best course of action in $\Sol$ seems to be to find a Daniel type correspondence.
	
	\appendix
	\section{ODE for horizontal \texorpdfstring{\mcH}{mcH}-cylinders in \texorpdfstring{$\Sol$}{Sol}}
	We compute the ODE for horizontal \mcH-cylinders in $\Sol$. We used this explicit representation to compute the examples in \autoref{chapter12}. 
	
	The mean curvature of an surface $f$ invariant by left translations along the base in $\Sol$ is easy to compute in terms of the orthonormal frame \eqref{onf} from \autoref{chapter11}:
	\begin{Prop}\label{odesol}Let $f$ be as in \eqref{surface}, i.e., $f$ parametrises a surface invariant by left translations along the base $c$ in $\Sol$. Then we have \[C:=\sqrt{\det(g)}=\sqrt{{x^\prime}^2+{y^\prime}^2+(x^\prime y+xy^\prime)^2}\]
		for the induced Riemannian metric $g$ on $\R\times J$. Moreover the mean curvature $H$ of $f$ in terms of $\gamma=(x,y,0)$ with respect to the inner normal satisfies the equation
		\begin{equation}
		\begin{split}
		2HC^3&= \left[xy^\prime-x^\prime y+(x^2-y^2)(xy^\prime +x^\prime y)\right]\cdot\left[{x^\prime}^2+{y^\prime}^2\right]\\
		&\hphantom{=}+2(yy^\prime +xx^\prime)(yy^\prime-xx^\prime)(xy^\prime +x^\prime y)\\
		&\hphantom{=}+(x^2+y^2+1)\left(x^\prime y^{\prime\prime}-x^{\prime\prime}y^\prime+({x^\prime}^2-{y^\prime}^2)(xy^\prime+x^\prime y)\right).
		\end{split}\label{odecurve}
		\end{equation} 
	\end{Prop}
	
	\begin{proof}[Sketch of proof]We have
		\[v_1:=\frac{\partial f}{\partial s}(s,t)=\begin{pmatrix}
		-e^{-s}x(t)\\ e^sy(t)\\1
		\end{pmatrix}=-xE_1+yE_2+E_3\]
		and
		\[v_2:=\frac{\partial f}{\partial t}(s,t)=\begin{pmatrix}
		e^{-s}x^\prime(t)\\e^sy^\prime(t)\\ 0
		\end{pmatrix}=x^\prime E_1+y^\prime E_2.\]
		
		Thus the inner normal $N$ to $f$ is
		\begin{equation}N=\frac{1}{\sqrt{{x^\prime}^2+{y^\prime}^2+(x^\prime y+xy^\prime)^2}}\left[-y^\prime E_1+x^\prime E_2-(xy^\prime +x^\prime y)E_3\right].\label{normal}
		\end{equation}
		
		The entries of the induced metric $g=(\langle v_j,v_k\rangle)_{1\leq j,k\leq 2}$ on $\R\times J$ are
		\begin{equation}
		\begin{split}
		g_{11}&=x^2+y^2+1,\\
		g_{12}&=-xx^\prime+yy^\prime,\\
		g_{22}&={x^\prime}^2+{y^\prime}^2.\end{split}\label{ff}
		\end{equation}
		Furthermore let us compute $\nabla_{v_j}v_k$ for $j,k\in\{1,2\}$:
		\begin{align*}
		\nabla_{v_1}v_1&=-xE_1-yE_2+(y^2-x^2)E_3,\\
		\nabla_{v_1}v_2&=(xx^\prime+yy^\prime)E_3,\\
		\nabla_{v_2}v_2&=x^{\prime\prime}E_1+y^{\prime\prime}E_2+\left({y^\prime}^2-{x^\prime}^2\right)E_3.
		\end{align*}
		
		It can be checked that $C:=\sqrt{\det(g)}$ agrees with the denominator of the coefficients in \eqref{normal}, i.e. we have \[C\cdot N=-y^\prime E_1+x^\prime E_2-(xy^\prime+x^\prime y)E_3.\]
		Thus the second fundamental form $b=\left(\langle \nabla_{v_j}v_k,N\rangle\right)_{1\leq j,k\leq 2}$ satisfies:
		\begin{align*}
		C b_{11}&=xy^\prime-x^\prime y+(x^2-y^2)(xy^\prime+x^\prime y),\\
		C b_{12}&=-(xx^\prime+yy^\prime)(xy^\prime+x^\prime y),\\
		C b_{22}&=-x^{\prime\prime}y^\prime+x^\prime y^{\prime\prime}+({x^\prime}^2-{y^\prime}^2)(xy^\prime+x^\prime y).
		\end{align*}
		In order to verify \eqref{odecurve}, the previous expressions must be plugged into
		\begin{align*}
		2H&=\frac{g_{22}}{C^2}b_{11}-2\frac{g_{12}}{C^2}b_{12}+\frac{g_{11}}{C^2}b_{22}=\frac{g_{22} C b_{11}-2g_{12} C b_{12}+g_{11} C b_{22}}{C^3}.\qedhere
		\end{align*}
	\end{proof}
	
	%
	%
	%
	%
	%
	%
	%
	%
	
	
	\subsubsection*{Acknowledgement}This work is an extended part of the author's Ph.D. thesis at TU Darmstadt. The author would like to thank his advisor Karsten Gro\ss{}e-Brauckmann for guidance and suggestions throughout the preparation of this paper, and the referee for a thorough report resulting in significant improvements.
	\bibliographystyle{amsalpha} 
	\bibliography{cylinder}

\providecommand{\bysame}{\leavevmode\hbox to3em{\hrulefill}\thinspace}
\providecommand{\MR}{\relax\ifhmode\unskip\space\fi MR }
\providecommand{\MRhref}[2]{%
  \href{http://www.ams.org/mathscinet-getitem?mr=#1}{#2}
}
\providecommand{\href}[2]{#2}
\begin{thebibliography}{MMPR17}

\bibitem[AR04]{abresch2004}
Uwe Abresch and Harold Rosenberg, \emph{A hopf differential for constant mean
  curvature surfaces in $\mathbb{S}^2\times\mathbb{R}$ and
  $\mathbb{H}^2\times\mathbb{R}$}, Acta Math. \textbf{193} (2004), no.~2,
  141--174.

\bibitem[AR05]{Abresch}
Uwe Abresch and Harold Rosenberg, \emph{Generalized {H}opf {D}ifferentials},
  Mat. Contemp. \textbf{28} (2005), 1--28.

\bibitem[Dan07]{Daniel}
Beno\^{\i}t Daniel, \emph{Isometric immersions into $3$-dimensional homogeneous
  manifolds}, Commentarii Mathematici Helvetici \textbf{82} (2007), no.~1,
  87--131.

\bibitem[DM13]{DanielMira}
Beno\^{\i}t Daniel and Pablo Mira, \emph{Existence and uniqueness of constant
  mean curvature spheres in $\operatorname{Sol}_3$}, Journal f\"ur die reine
  und angewandte Mathematik \textbf{685} (2013), 1--32.

\bibitem[DMH09]{Notes2009}
Beno\^{\i}t Daniel, Pablo Mira, and Laurent Hauswirth, \emph{Lecture notes on
  homogeneous $3$-manifolds}, 4th KIAS workshop on Differential Geometry, Seoul
  (2009).

\bibitem[Eng06]{Engel}
Sven Engel, \emph{On the geometry and trigonometry of homegeneous 3-manifolds
  with 4-dimensional isometry group}, Mathematische Zeitschrift \textbf{254}
  (2006), no.~3, 439--459.

\bibitem[FMP99]{FMP}
Chrstiam~B. Figueroa, Francesco Mercuri, and Renato H.~L. Pedrosa,
  \emph{Invariant {S}urfaces of the {H}eisenberg {G}roups}, Annali di
  Matematica pura ed applicata \textbf{CLXXVII} (1999), no.~IV, 173--194.

\bibitem[HdLR05]{HLR}
David Hoffman, Jorge~S.H. de~Lira, and Harold Rosenberg, \emph{Constant {M}ean
  {C}urvature {S}urfaces in ${M}^2\times\mathbb{R}$}, Transactions of the
  American Mathematical Society \textbf{358} (2005), no.~2, 491--507.

\bibitem[HH89]{Hsiang}
Wu-Teh Hsiang and Wu-Yi Hsiang, \emph{On the uniqueness of isoperimetric
  solutions and imbedded soap bubbles in non-compact symmetric spaces}, Invent.
  math. \textbf{98} (1989), 39--58.

\bibitem[Lop14]{Lopez14}
Rafael Lopez, \emph{Invariant surfaces in $\operatorname{Sol}_3$ with constant
  mean curvature and their computer graphics}, Advances in Geometry \textbf{14}
  (2014), no.~1.

\bibitem[Maz13]{Mazet13}
Laurent Mazet, \emph{A general halfspace theorem for constant mean curvature
  surfaces}, American Journal of Mathematics \textbf{135} (2013), no.~3,
  801--834.

\bibitem[Mee13]{MeeksSphere}
William Meeks, \emph{Constant {M}ean {C}urvature {S}pheres in
  $\operatorname{Sol}_3$}, American Journal of Mathematics (2013), no.~135.

\bibitem[MMP17]{MeeksTransversality}
William Meeks, III, Pablo Mira, and Joaqu\'{i}­n P\'{e}rez, \emph{Embeddedness
  of spheres in homogeneous three-manifolds}, International Mathematics
  Research Notices \textbf{2017} (2017), no.~15, 4796--4813.

\bibitem[MMPR14]{MeeksCheeger}
William Meeks, III, Pablo Mira, Joaqu\'{i}­n P\'{e}rez, and Antonio Ros,
  \emph{Isoperimetric domains of large volume in homogeneous three-manifolds},
  Advances in Math. (2014), no.~264, 546--592.

\bibitem[MMPR17]{Meeks2017}
\bysame, \emph{Constant mean curvature spheres in homogeneous three-manifolds},
  ArXiv e-prints (2017).

\bibitem[MP12]{MP}
William~H. Meeks and Joaqu\'{i}n P\'{e}rez, \emph{Constant {M}ean {C}urvature
  {S}urfaces in metric lie groups}, Contemporay Mathematics \textbf{570}
  (2012), 25--110.

\bibitem[MT14]{ManzanoTorralbo}
Jos\'e~M. Manzano and Francisco Torralbo, \emph{New {E}xamples of {C}onstant
  {M}ean {C}urvature in {$\mathbb{S}^2\times\mathbb{R}$} and
  {$\mathbb{H}^2\times\mathbb{R}$}}, Michigan Math J. \textbf{63} (2014),
  no.~4, 701--723.

\bibitem[Onn08]{Onnis}
Irene~I. Onnis, \emph{Invariant surfaces with constant mean curvature in
  $\mathbb{H}^2\times\mathbb{R}$}, Annali di Matematica \textbf{187} (2008),
  667--682.

\bibitem[Pen10]{Penafiel}
Carlos Penafiel, \emph{Surfaces of constant mean curvature in homogeneous three
  manifolds with emphasis in $\operatorname{PSL}_2(\mathbb{R},\tau)$}, Ph.D.
  thesis, PUC-RJ, Brazil, 2010.

\bibitem[ST05]{SaEarp}
Ricardo {Sa Earp} and Eric Toubiana, \emph{Screw motion surfaces in
  $\mathbb{H}^2\times\mathbb{R}$ and $\mathbb{S}^2\times\mathbb{R}$}, Illinois
  J. Math \textbf{49} (2005), 1323--1362.

\bibitem[Tor10]{Torralbo}
Francisco Torralbo, \emph{Rotationally invariant constant mean curvature
  surfaces in homogeneous $3$-manifolds}, Differential Geometry and its
  Applications \textbf{28} (2010), no.~5, 593--607.

\bibitem[Vrz17]{Vrzina}
Miroslav Vrzina, \emph{On the existence problem for tilted unduloids in
  $\mathbb{H}^2\times\mathbb{R}$}, ArXiv e-prints (2017).

\end{thebibliography}
	
	
\end{document}